\documentclass{siamltex}
\usepackage{amsmath,amsfonts,amssymb}
\usepackage{latexsym}
\usepackage{graphicx,overpic}
\usepackage{subfigure,caption}
\usepackage{pgfplots} 
\pgfplotsset{compat=newest} 
\usepackage{color}
\usepackage{booktabs}
\usepackage{tikz}

\newcommand{\R}{\mathbb R}
\newcommand{\RR}{\mathbb R}

\newcommand{\bH}{\mathbf H}
\newcommand{\bI}{\mathbf I}

\newcommand{\bP}{\mathbf P}

\newcommand{\bg}{\mathbf g}
\newcommand{\bn}{\mathbf n}

\newcommand{\bp}{\mathbf p}

\newcommand{\bu}{\mathbf u}
\newcommand{\bv}{\mathbf v}

\newcommand{\bx}{\mathbf x}

\newcommand{\bz}{\mathbf z}

\newcommand{\blf}{\mathbf f}
\newcommand{\bL}{\mathbf L}

\newcommand{\divG}{{\mathop{\,\rm div}}_{\Gamma}}

\newcommand{\gradG}{\nabla_{\Gamma}}

\newcommand{\curlG}{{\mathop{\,\rm curl}}_{\Gamma}}
\newcommand{\vcurlG}{{\mathop{\,\rm \mathbf{curl}}}_{\Gamma}}

\newcommand{\cT}{\mathcal T}

\renewcommand{\div}{\textrm{div}\ \!}

\DeclareGraphicsExtensions{.pdf,.eps,.ps,.eps.gz,.ps.gz,.eps.Y}

\DeclareMathOperator*{\RotG}{\mathbf{curl}_\Gamma \!}

\DeclareMathOperator*{\RotGh}{\mathbf{curl}_{\Gamma_h} \!}

\DeclareMathOperator*{\Div}{div \!}
\DeclareMathOperator*{\DivG}{div_\Gamma \!}

\DeclareMathOperator*{\nablaG}{\nabla_\Gamma \!}

\DeclareMathOperator*{\nablaGh}{\nabla_{\Gamma_h} \!}

\DeclareMathOperator*{\DeltaG}{\Delta_\Gamma \!}

\DeclareMathOperator*{\tr}{tr}

\DeclareMathOperator*{\interior}{int}

\newtheorem{assumption}{Assumption}[section]

\newtheorem{remark}{Remark}[section]

\begin{document}

\title{Finite Element Error Analysis of Surface Stokes Equations in Stream Function Formulation}
\author{
Philip Brandner\thanks{Institut f\"ur Geometrie und Praktische  Mathematik, RWTH-Aachen
University, D-52056 Aachen, Germany (brandner@igpm.rwth-aachen.de).}
\and
Arnold Reusken\thanks{Institut f\"ur Geometrie und Praktische  Mathematik, RWTH-Aachen
University, D-52056 Aachen, Germany (reusken@igpm.rwth-aachen.de).}
}
\maketitle
%
\begin{abstract} We consider a surface Stokes problem in stream function formulation on a simply connected oriented surface $\Gamma \subset \Bbb{R}^3$ without boundary. This formulation leads to a coupled system of two second order scalar surface partial differential equations (for the stream function and an auxiliary variable). To this coupled system a trace finite element discretization method is applied. The main topic of the paper is an error analysis of this discretization method, resulting in optimal order discretization error bounds. The analysis applies to the surface finite element method of Dziuk-Elliott, too.  We also investigate methods for reconstructing velocity and  pressure from the stream function approximation. Results of numerical experiments are included.
\end{abstract}
\section{Introduction}
In recent years there has been a strongly growing interest in the field of modeling and numerical simulation of surface fluids. Fluidic surfaces or fluidic interfaces are used in, for example, models describing  emulsions, foams or  biological  membranes; cf., e.g., \cite{scriven1960dynamics,slattery2007interfacial,arroyo2009,brenner2013interfacial, rangamani2013interaction,rahimi2013curved}. Typically such models consist of surface {(Navier-)}Stokes equations. These equations 
are also studied as an interesting mathematical problem in its own right in, e.g.,  \cite{ebin1970groups,Temam88,taylor1992analysis,arnol2013mathematical,mitrea2001navier, arnaudon2012lagrangian,Gigaetal}. Recently there has been a strong increase in research on numerical simulation methods for surface (Navier-)Stokes equations, e.g.,  \cite{nitschke2012finite,barrett2014stable,reuther2015interplay,Reusken_Streamfunction,reuther2017solving,fries2017higher,olshanskii2018finite,olshanskii2019penalty,Bonito2019a,OlshanskiiZhiliakov2019,Lederer2019}. By far most of these and other papers on numerical methods for surface (Navier-)Stokes equations these equations are treated in the primitive velocity and pressure variables. In the paper \cite{nitschke2012finite} a finite element discretization of the Navier-Stokes equations on a stationary smooth closed surface in \emph{stream function} formulation is presented. In \cite{Reusken_Streamfunction} a surface Helmholtz decomposition and  well-posedness of this stream function formulation for a class of surface Stokes problems are studied.  We are not aware of any other literature in which surface (Navier-)Stokes equations in stream function formulation are studied.

In Euclidean space, the stream function formulation of (Navier-)Stokes is well-known  and thoroughly studied, e.g., \cite{GR,Quarteroni} and the references therein. In numerical simulations of three-dimensional problems this formulation is not often used due to substantial disadvantages. For two-dimensional problems this formulation reduces to a fourth order biharmonic equation for the \emph{scalar} stream function. This formulation has been used in numerical simulations, although it has certain disadvantages related to boundary conditions and regularity (\cite{GR,Quarteroni}).  In the fields of applications mentioned above, one often deals with smooth simply connected surfaces without boundary. In such a setting there usually  are no difficulties related to regularity or boundary conditions and the stream function formulation may be a very attractive alternative  to the formulation in primitive variables, as already indicated in \cite{nitschke2012finite}.

In both papers \cite{nitschke2012finite,Reusken_Streamfunction} the resulting fourth order scalar surface partial differential equation for the stream function is reformulated as a coupled system of two second order equations, which is a straightforward generalization to surfaces of the classical Ciarlet-Raviart method \cite{Ciarlet1974} in Euclidean space. These equations can be discretized by established finite element methods for the discretization of scalar elliptic surface partial differential equations, such as the surface finite element method \cite{DEreview} (SFEM) or the trace finite element method \cite{olshanskii2016trace} (TraceFEM); cf. also the recent overview paper \cite{Bonito2019}. 

In this paper we present an error analysis of a finite element method for the discretization of the coupled system of second order surface PDEs for the stream function (and an auxiliary variable). For the Euclidean case (i.e., the Ciarlet-Raviart method) an error analysis is presented in the papers \cite{Falk1980,Babuska1980}. In these papers variants of the fundamental abstract Brezzi saddle point theory are developed that are applicable to the mixed saddle point formulation of the biharmonic equation. It turns out that the analyses presented in these classical papers are \emph{not} applicable to the mixed system that results from the stream function formulation of the \emph{surface} Stokes problem. The reason for this is the following. For the  stream function  
$\psi$ of the velocity solution $\bu$ of the surface Stokes equation, i.e. the relation $\bu =\vcurlG \psi$ holds, one obtains the  fourth order surface partial differential equation (in strong formulation)
\begin{equation} \label{eqIntro0}
  -\Delta_\Gamma^2 \psi - \divG(K \nablaG \psi) = \curlG \mathbf{f},
\end{equation}
with $\mathbf{f}$ the force term in the Stokes equation, $\Delta_\Gamma$ the Laplace-Beltrami operator  and $K$ the Gaussian curvature. Precise definitions of the curl operators $\vcurlG$ and  $\curlG$ are given in the next section. Introducing the auxiliary variable $\phi = \Delta_\Gamma \psi$ one obtains a coupled system of second order equations
\begin{equation} \label{eqIntro}
 \begin{split}
  \phi - \Delta_\Gamma \psi & = 0 \\
  -\Delta_\Gamma \phi - \divG(K \nablaG \psi) &= \curlG \mathbf{f}.
 \end{split}
\end{equation}
For the case $K=0$ (Euclidean domain) this problem has a standard saddle point structure, which is a structural property that is essential for the analyses in \cite{Falk1980,Babuska1980}. For the surface case, however, the coupling term $ \divG(K \nablaG \psi)$ destroys this nice saddle point structure. Note that the additional coupling term is second order and in general indefinite (because $K$ does not necessarily have a fixed sign).   It is not a-priori clear whether discretization methods such as the SFEM or TraceFEM, which have optimal order discretization errors for the Laplace-Beltrami equation, are of optimal order when applied to this coupled system. 

In this paper we present an error analysis of the TraceFEM applied to \eqref{eqIntro} and prove that this method has optimal order discretization errors. Our analysis  applies (with minor modifications) to the SFEM, too. We also introduce and analyze methods for reconstructing the velocity and pressure from the resulting finite element stream function approximation $\psi_h \approx \psi$. To avoid many technical details, we do not take  the so-called geometry errors, i.e., errors caused by the  approximation of the surface $\Gamma$, into account. Our analysis is inspired by the technique used in \cite{Falk1980}. We add a new key ingredient that is based on a relation between the fourth order problem \eqref{eqIntro0} for the stream function and the corresponding Stokes equation. This relation yields the result formulated in Corollary~\ref{cor_NotesAR_1}, which turns out to be sufficient to control the coupling term $ \divG(K \nablaG \psi)$ in the error analysis.

The remainder of the paper is organized as follows. In Section~\ref{sectSurfaceStokesStreamfunction} we introduce surface differential operators, recall results for the stream function formulation from \cite{Reusken_Streamfunction} and introduce a well-posed weak formulation of the coupled system \eqref{eqIntro}.  A trace finite element discretization for this coupled system is explained in Section~\ref{sectTraceFEM}. In Section~\ref{Sectdiscreconstr} we introduce finite element methods for the reconstruction of the velocity and pressure unknowns, given a stream function approximation. The main new results of this paper are presented in Section~\ref{SectAnalysisStreamfunction}. In that section an error analysis of the discretization method for the coupled system is derived. In Section~\ref{SectPressureVelocity} we analyze the discretization errors in the reconstruction procedures for velocity and pressure. Finally, in Section~\ref{SectExp} we present results of  numerical experiments that illustrate relevant 
properties of the discretization method. 

\section{Surface Stokes in stream function formulation} \label{sectSurfaceStokesStreamfunction}
\subsection{Preliminaries}
We consider a sufficiently smooth closed simply connected compact surface $\Gamma \subset \Bbb{R}^3$. The  signed distance function and outward pointing unit normal are denoted by $d$ and $\bn$, respectively. We define the closest point projection by $\bp (x)=x - d(x)\bn (x)$ on a sufficiently small neighborhood  of $\Gamma$. The orthogonal projection $\bP(x)=\bI - \bn(x)\bn(x)^T$, $x \in \Gamma$ is used. We introduce the following tangential derivative of a scalar function $\phi \in C^1(\Gamma)$ and of a vector function $\bu \in C^1(\Gamma)^3$ for $x \in \Gamma$:
\begin{align} \label{defgrad} 
\gradG \phi(x)& := \nabla(\phi \circ \bp)(x)= \bP(x)\nabla \phi^e(x),  \\
\gradG \bu(x)& :=\bP(x)\left( \frac{\partial (\bu\circ \bp)(x)}{\partial x_1}~ \frac{\partial( \bu \circ \bp)(x)}{\partial x_2} ~\frac{\partial( \bu \circ \bp)(x)}{\partial x_3}\right)  \nonumber \\ & = \bP(x)\nabla\bu^e(x)\bP(x). \label{covder}
\end{align}
Here, $\phi^e$, $\bu^e$ denote some smooth extension of $\phi$ and $\bu$ on the neighborhood $U$, and $\nabla\bu^e$ is the Jacobian, $(\nabla\bu^e)_{i,j}=\frac{\partial u^e_i}{\partial x_j}$, $1\leq i,j\leq 3$. In the remainder the argument $x \in \Gamma$ is deleted. Tangential divergence operators  are defined by
 \begin{equation} \label{defDivG}
  \DivG \bu:= \tr(\gradG \bu), \quad\DivG A:= \begin{pmatrix} \DivG(e_1^TA) \\ \DivG( e_2^TA) \\ \DivG (e_3^T A) \end{pmatrix}, \quad A\in C^1(\Gamma)^{3 \times 3}, 
 \end{equation}
with $e_i$, $i=1,2,3$ the standard basis vectors in $\Bbb{R}^3$. We define the following surface curl operators:
 \begin{align}
\label{defdivcurl}
  \curlG \bu= \DivG(\bu \times \bn), \quad \bu \in C^1(\Gamma)^3,\\
\label{defdivcurl2}
  \vcurlG \phi:= \bn \times \gradG \phi, \quad \phi \in C^1(\Gamma).
\end{align}

\subsection{Surface Stokes problem in stream function formulation} \label{sectStokescurl}
Motivated by the modeling of surface fluids studied in e.g., \cite{GurtinMurdoch75,barrett2014stable,Gigaetal,jankuhn,Miura}, we introduce for a given given parameter $\alpha \geq 0$ and force vector $\mathbf{f} \in L^2(\Gamma)^3$, with $\mathbf{f}\cdot\bn=0$ the following  surface Stokes problem: Determine a  tangential velocity vector field $\bu:\, \Gamma \to \R^3$, with $\bu\cdot\bn =0$, and the surface fluid pressure $p$ such that
\begin{align} 
 - \bP \DivG (E_s(\bu)) + \alpha \bu +\nabla_\Gamma p &=  \blf \quad \text{on}~\Gamma,  \label{strongform-1} \\
 \DivG \bu & =0 \quad \text{on}~\Gamma. \label{strongform-2}
\end{align}
Here, the surface rate-of-strain tensor $E_s(\bu):= \frac12(\nabla_\Gamma \bu + \nabla_\Gamma \bu^T)$ is used.  The pressure field is defined up to a hydrostatic mode and all tangentially rigid surface fluid motions, i.e. satisfying  $E_s(\bu)=0$, are called \textit{Killing vector fields} \cite{sakai1996riemannian}. For the case $\alpha=0$ one needs the additional consistency condition $\int_\Gamma \blf \cdot \bv\,ds=0$ for all smooth  Killing vector fields $\bv$, which follows   from integration by parts.\\
For the weak formulation of the problem \eqref{strongform-1}-\eqref{strongform-2} we introduce the spaces of \emph{tangential} vector functions
\[ \begin{split}
 \bL_t^2(\Gamma)& :=\{ \, \bu \in L^2(\Gamma)^3~|~ \bn \cdot \bu=0 \quad \text{a.e. on}~\Gamma\,\},\\  \bH_t^1(\Gamma) &:= \{ \, \bu \in H^1(\Gamma)^3~|~ \bn \cdot \bu=0 \quad \text{a.e. on}~\Gamma\,\}, \\  \bH_{t,\div}^1 & := \{\, \bu \in \bH_t^1(\Gamma)~|~\DivG \bu =0 \,\}.
\end{split} \]
The space of Killing vector fields is denoted by
\begin{equation}   \label{defVT}
 E:= \{\, \bu \in \bH_t^1(\Gamma)~|~ E_s(\bu)=0\,\}.
\end{equation}
The space $E$ is a closed subspace of $\bH_t^1(\Gamma)$ and $\mbox{dim}(E)\le 3$ holds. The usual bilinear forms are introduced:
\begin{align}
a_\alpha(\bu,\bv)& := \int_\Gamma E_s(\bu):E_s(\bv) + \alpha \bu \cdot \bv\, ds, \quad \bu,\bv \in \bH_t^1(\Gamma), \label{defblfa} \\
b(\bv,p) &:= - \int_\Gamma p\,\DivG \bv \, ds,  \quad \bv \in \bH_t^1(\Gamma), ~p \in L^2(\Gamma). \label{defblfb}
\end{align}
For the case $\alpha=0$ and $\bu \in E$  we have $a_\alpha(\bu,\bv)=0$ for all $\bv \in \bH_t^1(\Gamma)$  and therefore this kernel space $E$ has to be factored out. For this we introduce the notation $E_0:=E$ and $E_\alpha = \emptyset$ for $\alpha >0$.   
We consider the variational formulation of the surface Stokes problem \eqref{strongform-1}-\eqref{strongform-2}: Determine $(\bu,p) \in  \bH_t^1(\Gamma)/E_\alpha \times L_0^2(\Gamma)$ such that
 \begin{equation} \label{Stokesweak1_1}
 \begin{split}
           a_\alpha(\bu,\bv) +b(\bv,p) &=(\blf,\bv)_{L^2(\Gamma)} \quad \text{for all}~~\bv \in  \bH_t^1(\Gamma)/E_\alpha,  \\
           b(\bu,q) & =0 \qquad \text{for all}~~q \in L^2(\Gamma). 
 \end{split}
 \end{equation}

 From a surface Korn inequality for $a_\alpha(\cdot,\cdot)$, an inf-sup property for $b(\cdot,\cdot)$ (\cite{jankuhn}), and the continuity of the bilinear forms $a_\alpha(\cdot,\cdot)$ and $b(\cdot,\cdot)$ one obtains the \emph{well-posedness} of the problem \eqref{Stokesweak1_1}. Its unique solution is denoted by $\{\bu^\ast,p^\ast\}$. Note that the unique solution $\bu^\ast$ is also the unique solution of the following problem: determine $\bu \in \bH_{t,\Div}^1/E_\alpha$ such that 
\begin{equation} \label{Stokes3}
  a_\alpha(\bu,\bv)=(\blf,\bv)_{L^2(\Gamma)} \quad \text{for all}~~\bv \in \bH_{t,\Div}^1/E_\alpha.
\end{equation}
The stream function formulation of the  surface Stokes problem is based on a surface Helmholtz decomposition and  the identity given in the next lemma.
\begin{lemma} \label{lemma1}
The following relation holds for all $\phi,\psi \in H^2(\Gamma)$:
\begin{equation} \label{identa}
 \begin{split} a_\alpha (\vcurlG \phi,\vcurlG \psi)& = \int_{\Gamma} E_s(\vcurlG \phi):E_s(\vcurlG \psi) +\alpha\vcurlG \phi \cdot \vcurlG \psi \, ds  \\ & =
 \int_\Gamma \frac12 \DeltaG \phi \DeltaG \psi + (\alpha- K)\gradG \phi \cdot \gradG \psi \, ds \\ & =:\tilde a_\alpha (\phi,\psi).
 \end{split}
 \end{equation}
 Here, $K$ denotes the Gaussian curvature of the surface $\Gamma$.
\end{lemma}
\begin{proof}
For $\alpha=0$ a proof is given in Lemma 5.3. of \cite{Reusken_Streamfunction}. This proof can easily be extended  to the case $\alpha>0$.
\end{proof}
\ \\[1ex] 
The following spaces are used in the stream function formulation:
\begin{align*}
 H_\ast^k(\Gamma)& :=\{\, \psi \in H^k(\Gamma)~|~\int_\Gamma \psi \, ds =0\,\}, \quad \tilde E := \{\, \psi \in H_\ast^2(\Gamma)~|~\tilde a_\alpha(\psi,\psi)=0\,\}.
\end{align*}
Analogous to $E_\alpha$ we define $\tilde E_0:=\tilde E$, $\tilde E_\alpha := \emptyset$ for $\alpha >0$.
For the  stream function formulation the result in the following lemma is essential. For this result to hold, \emph{the assumption that $\Gamma$ is simply connected is necessary}.
\begin{lemma} \label{lemma2}
Assume that $\Gamma$ is simply connected. The following holds:
 \begin{align*}
  \vcurlG &:\, H_\ast^2(\Gamma) \to \bH_{t,\Div}^1 \quad \text{is an homeomorphism,} \\ 
  \vcurlG &:\,\tilde E  \to E \quad \text{is an homeomorphism}. 
 \end{align*}
\end{lemma}
\begin{proof}
A proof is given in Lemma 5.4. of \cite{Reusken_Streamfunction}.
\end{proof}
\ \\[1ex]
The following theorem introduces the stream function formulation of the  surface Stokes problem. 
\begin{theorem} \label{thmstream}
 Let $\bu^\ast \in \bH_{t,\Div}^1/E_\alpha$ be the unique solution of \eqref{Stokesweak1_1} {\rm(}or \eqref{Stokes3}{\rm)} and $\psi^\ast \in H_\ast^1(\Gamma)$  the unique stream function such that $\bu^\ast= \vcurlG \psi^\ast$. This $\psi^\ast$ is the unique solution of the following problem: determine $\psi \in H_\ast^2(\Gamma)/\tilde E_\alpha$ such that
 \begin{equation} \label{weakstream}
  \tilde a_\alpha (\psi,\phi)= (\blf, \vcurlG \phi)_{L^2(\Gamma)} \quad \text{for all}~~\phi \in H_\ast^2(\Gamma)/\tilde E_\alpha.
 \end{equation}
 Furthermore,  the regularity estimate
 \begin{equation} \label{regestimate}
  \|\psi^\ast\|_{H^3(\Gamma)} \leq c \|\blf\|_{L^2(\Gamma)}
 \end{equation}
holds, with a constant $c$ independent of $\blf \in \bL_t^2(\Gamma)$.
\end{theorem}
\begin{proof}
Again for $\alpha=0$ a proof is given in Theorem 5.5. of \cite{Reusken_Streamfunction}. This proof can easily be extended to the case $\alpha>0$.
\end{proof}
\ \\[1ex]
The main topic of this paper is an error analysis of a finite element method for the discretization of the stream function formulation \eqref{weakstream}. We reformulate the fourth order surface problem \eqref{weakstream} as a coupled system of two second order equations, which are then discretized using a specific  finite element method. Before we present this coupled system we collect the
(main) assumptions concerning the stream function formulation.
\begin{assumption} \label{ass} We assume that $\Gamma$ is simply connected and sufficiently smooth, at least $C^3$. We take $\alpha=1$.
\end{assumption}

The assumption that $\Gamma$ is simply connected is essential for the stream function formulation. We do not give precise statements on how the results derived in the error analysis below depend on the smoothness of $\Gamma$. For \eqref{regestimate} to hold we need that $\Gamma$ is at least $C^3$. To avoid technical details related to the Killing vector fields we restrict to the case $\alpha =1$. The results derived in the paper also hold (with minor modifications) for the case $\alpha=0$.

For a reformulation of \eqref{weakstream} as a coupled system we introduce the following \emph{symmetric} bilinear forms:
\begin{align}
m(\xi,\eta) &:= \int_{\Gamma} \xi \eta \, ds, \label{defm} \\
b(\xi,\eta) &:= \int_{\Gamma} \nablaG \xi \cdot  \nablaG \eta \, ds, \label{defb} \\
b_K(\xi,\eta) &:= 2 \int_{\Gamma} (1-K) \nablaG \xi \cdot  \nablaG \eta \, ds, \label{defbK}
\end{align}
and the functional
\begin{align} \label{rhsg}
g(\xi) &:= -2 \int_{\Gamma} \blf \cdot \RotG \xi \, ds.
\end{align}
Note that the notation $b(\cdot,\cdot)$ is already used in the variational formulation of the Stokes problem. We use the notation $b(\cdot,\cdot)$ in \eqref{defb}, because it  is consistent to the notation used in the paper \cite{Falk1980}, from which main ideas  of our error analysis are taken. In the remainder only the bilinear form $b(\cdot,\cdot)$ as defined in \eqref{defb} is used, hence,  confusion  is avoided.

The coupled formulation is as follows: Determine $\psi \in H^1_*(\Gamma)$, $\phi \in H^1(\Gamma)$ such that
\begin{equation}
\label{NotesAR_8}
\begin{split}
m(\phi,\eta) + b(\psi, \eta) &= 0 \quad \text{for all} ~ \eta \in H^1(\Gamma),\\
b(\phi,\xi) - b_K(\psi, \xi) &= g(\xi) \quad \text{for all} ~ \xi \in H^1(\Gamma).
\end{split}
\end{equation}
\begin{lemma} \label{lemreformulation}
The problem \eqref{NotesAR_8} has a unique solution  $\psi=\psi^\ast,~ \phi=\phi^\ast = \DeltaG \psi^\ast$, with $\psi^\ast$ the unique solution of \eqref{weakstream}. Furthermore, we have the regularity estimates
\begin{align}
\label{NotesAR_9}
\| \psi^\ast \|_{H^3(\Gamma)} \leq c \|\blf\|_{L^2(\Gamma)}, \quad \| \phi^\ast \|_{H^1(\Gamma)} \leq c \|\blf\|_{L^2(\Gamma)}.
\end{align}
\end{lemma}
\begin{proof}
For $\alpha=0$ a proof is given in Lemma 5.6. of \cite{Reusken_Streamfunction}. This proof also applies to the case $\alpha=1$.
\end{proof}
\ \\[1ex]
Let the stream function $\psi^\ast$ be the unique solution of \eqref{NotesAR_8}. We now introduce (obvious) variational formulations that are useful for the reconstruction of the solution $\{\bu^\ast,p^\ast\}$ of the surface Stokes problem, given $\psi^\ast$. First we consider $\bu$.
By definition we have
\begin{align*}
\bu^\ast =\RotG \psi^\ast = \bn \times \nablaG \psi^\ast.
\end{align*}
This immediately leads to the following well-posed variational formulation: Determine $\bu \in L^2(\Gamma)^3$ such that
\begin{align}
\label{Variationsformulierung_u}
\int_{\Gamma} \bu \cdot \bv \text{ } ds = \int_{\Gamma} \left( \bn \times \nablaG \psi^\ast \right) \cdot \bv \text{ } ds =: \int_{\Gamma} \bg \cdot \bv \text{ } ds \quad \text{for all}~\bv \in L^2(\Gamma)^3,
\end{align}
which has the unique solution $\bu^\ast$. Below, in the discretization method we use a finite element discretization of \eqref{Variationsformulierung_u}. Note that from regularity theory for the Stokes problem we know that $\bu^\ast$ has more regularity than only $\bu^\ast \in L^2(\Gamma)^3$. A regularity relation  between the velocity solution $\bu^\ast$ and its corresponding stream function $\psi^\ast$ can be derived. Assume $\bu^\ast \in H^{r}_*(\Gamma)$ holds. Then the inequalities
\begin{align}
\label{NormConnection_U_Psi}
c_1 \| \psi^\ast \|_{H^{r+1}(\Gamma)} \leq \| \bu^\ast \|_{H^r(\Gamma)} \leq c_2 \| \psi^\ast \|_{H^{r+1}(\Gamma)} 
\end{align}
hold, with constants $c_1, \, c_2 >0$ that depend only on $\Gamma$.  This can be derived as follows. For the first estimate we use a Poincare inequality $\|\psi\|_{H^{r+1}(\Gamma)} \leq c \| \nabla_\Gamma\psi \|_{H^r(\Gamma)}$ for all $\psi \in H^{r+1}_\ast(\Gamma)$ and  the identity $\nablaG \psi^\ast = \bu^\ast \times \bn$. The second estimate follows directly from the identity $\bu^\ast = \bn \times \nablaG \psi^\ast$.

For the derivation of a pressure reconstruction  we take the first equation in \eqref{Stokesweak1_1} and insert a test function $\bv=\nabla_\Gamma \xi$, $\xi \in H^2(\Gamma)$.  Using the identity (cf. \cite{Reusken_Streamfunction})
\[
  \bP \divG ( E_s(\bu)) = \frac12\vcurlG (\curlG \bu) +  K \bu
\]
for any $\bu$ that satisfies $  \divG \bu =0$, we obtain for the solution $\{\bu^\ast, p^\ast\}$ of   \eqref{Stokesweak1_1}:
\begin{align*}
  (\blf,\gradG \xi)_{L^2(\Gamma)} =&    \int_{\Gamma} E_s(\bu^\ast):E_s(\gradG \xi) + \bu^\ast \cdot\gradG \xi +\gradG p^\ast \cdot \gradG \xi \, ds  \\  
   &= \int_\Gamma - \bP \divG \big( E_s(\bu^\ast) \big) \cdot \gradG \xi + \gradG p^\ast \cdot \gradG \xi\, ds \\ 
   & =  \int_\Gamma - \frac12 \vcurlG(\curlG \bu^\ast)\cdot \gradG \xi - K \bu^\ast  \cdot \gradG \xi + \gradG p^\ast \cdot \gradG \xi \, ds \\
   & = \int_\Gamma  - K \bu^\ast  \cdot \gradG \xi + \gradG p^\ast \cdot \gradG \xi \, ds,
\end{align*}
where in the last equality we used partial integration and $\divG(\vcurlG(\cdot))=0$. 
Thus with $\bu^\ast = \RotG \psi^\ast$ we obtain that $p^\ast$ is the unique solution of the following well-posed Laplace-Beltrami problem: Determine $p \in H^1_\ast(\Gamma)$ such that
\begin{equation}
\label{Variationsformulierung_p} \begin{split}
\int_\Gamma \nabla_\Gamma p \cdot \nabla_\Gamma \xi \text{ } ds  & =  \int_\Gamma \left(  K \RotG \psi^\ast + \blf  \right) \cdot \nabla_\Gamma \xi \text{ } ds  \\ & =: \int_\Gamma \bz \cdot \nabla_\Gamma \xi \text{ } ds \quad \text{for all}~ \xi \in H^1(\Gamma).
\end{split}
\end{equation}
Below, for the pressure reconstruction we will apply a  finite element discretization method to this Laplace-Beltrami problem.
\section{Trace finite element method for discretization of the coupled problem} \label{sectTraceFEM}
For the discretization of \eqref{NotesAR_8} we propose a  trace finite element method (TraceFEM). Alternatively, the surface finite element of Dziuk-Elliott \cite{DEreview} can be used, cf. Remark~\ref{RemDE}.

We outline the key ingredients of the (higher order) TraceFEM. More detailed explanations are given in \cite{lehrenfeld2016high,GrandeLehrenfeldReusken}. 

Let $\Omega \subset \RR^3$ be a polygonal domain that strictly contains the manifold $\Gamma$.  We need a sufficiently accurate approximation $\Gamma_h$ of $\Gamma$. For the definition of the method it is sufficient that $\Gamma_h$ is a Lipschitz surface without boundary. Possible constructions are briefly addressed in Remark~\ref{RemGammah}. We choose a shape regular family of tetrahedral triangulations 
$\{\cT_h\}_{h >0}$ of $\Omega$ and  introduce the active mesh
\begin{equation*}
\cT_h^\Gamma:=\{T \in \cT_h \,:\, {\rm meas}_2(T \cap \Gamma_h)>0  \},
\end{equation*}
consisting of the subset of all tetrahedra that have a nonzero intersection with $\Gamma_h$. The domain $\Omega^\Gamma_h:=\interior \left( \overline{ \cup_{T \in \cT_h^\Gamma} T } \right)$ is formed by the triangulation $\cT_h^\Gamma$. We introduce the standard (local) finite element space $V_{h,k}$ 
\begin{align*}
V_{h,k}:=  \{ v_h \in \mathcal{C}(\Omega^\Gamma_h) \,:\, {v_h}_{|T} \in \mathit{P}_k \quad \forall \, T \in \cT_h^\Gamma \}.
\end{align*}

\begin{remark} \label{RemGammah} \rm We assume that $\Gamma$ is (implicitly) represented as the zero level of a smooth level set function $\Phi$.
 We denote with $I^1_h$ the  nodal interpolation operator on $\cT_h$, which maps into the space of continuous piecewise linears on $\cT_h$.   A piecewise planar approximation $\Gamma_h$ of $\Gamma$ is given by 
\begin{equation} \label{eqG}
\Gamma_h:= \{ x \in \Omega : (I^1_h \Phi)(x)=0 \}.
\end{equation}
This $\Gamma_h$, which is easy to construct if $\Phi$ is available, has second order accuracy, i.e. ${\rm dist}(\Gamma_h,\Gamma) \lesssim h^2$. Here and  in the rest of the paper we use the notation $\lesssim$ to denote an inequality with a (hidden) constant that is independent of $h$ and of the position of $\Gamma$ in the mesh $\cT_h^\Gamma$. If a more accurate geometry approximation is required, one can replace the operator  $I^1_h$ by the nodal interpolation operator $I^q_h$, which  maps into the space of continuous piecewise polynomials of degree $q \geq 2$ on $\cT_h$. In that case, however, the zero level of $I^q_h$ is not easy to determine. It is better to use a variant of this approach in which, based on a sufficiently accurate approximation $\Phi_h$ of $\Phi$, a parametric mapping $\Theta_h^q \in (V_{h,q})^3$ is constructed, which deforms the local triangulation $\cT_h^\Gamma$ in such a way that $\Gamma_h^q:=\Theta_h^q(\Gamma_h)$ (with $\Gamma_h$ as in  \eqref{eqG}) has accuracy ${\rm dist}(\Gamma_
h^q,\Gamma) \lesssim h^{q+1}$. This parametric mapping induces corresponding parametric finite element spaces which are then used, instead of  $V_{h,k}$. A precise explanation of this 
parametric trace finite element method  and  an error analysis of this method are given in \cite{lehrenfeld2016high,GrandeLehrenfeldReusken}.
\end{remark}

We now introduce the bilinear forms used in the discretization of \eqref{NotesAR_8}. Besides natural discrete analogons (corresponding to $\Gamma_h$) of the bilinear forms used in  \eqref{NotesAR_8} we need an additional one, related to stabilization. It is well known that in the setting of TraceFEM one needs a suitable stabilization for damping instabilities caused by ``small cuts'' \cite{cutFEM}. For this we use the \emph{volume normal derivative stabilization}, known from the literature, denoted by $s_h(\cdot,\cdot)$ below. We define:
\begin{align*} 
m_h(\xi,\eta) &:= \int_{\Gamma_h} \xi \eta \, ds_h, \\
b_h(\xi,\eta) &:= \int_{\Gamma_h} \nablaGh \xi \cdot  \nablaGh \eta \, ds_h,\\ 
b_{h,K}(\xi,\eta) &:= 2 \int_{\Gamma_h} (1-K_h) \nablaGh \xi \cdot  \nablaGh \eta \, ds_h,\\
s_h(\xi,\eta) & := \rho \int_{\Omega^\Gamma_h} (\bn_h \cdot \nabla \xi) (\bn_h \cdot \nabla \eta)  \, dx, \\
g(\xi) &:= -2 \int_{\Gamma_h} \blf_h \cdot \RotGh \xi \, ds_h,
\end{align*}
with $K_h$ an approximation of the Gauss curvature $K$, $\bn_h$ the normal on $\Gamma_h$ and $ \blf_h \approx \blf $ a data extension. For the stabilization parameter $\rho$ we restrict to the usual  range \cite{GrandeLehrenfeldReusken} 
\begin{equation} \label{rhoscaling}
  h \lesssim \rho \lesssim h^{-1}.
\end{equation}

We consider the following discretization of \eqref{NotesAR_8}: Determine $\phi_h \in V_{h,k}$ and $\psi_h \in V_{h,k}$ with $\int_{\Gamma_h} \psi_h \, ds_h= 0$,  such that
\begin{equation}\label{NotesAR_10}
 \begin{split}
m_h(\phi_h,\eta_h) + b_h(\psi_h, \eta_h) + s_h(\psi_h, \eta_h) &= 0 \quad \text{for all}~ \eta_h \in V_{h,k},\\
b_h(\phi_h,\xi_h) + s_h(\phi_h, \xi_h)- b_{h,K}(\psi_h, \xi_h)  &= g_h(\xi_h) \quad \text{for all}~\xi_h \in V_{h,k}.
\end{split}\end{equation}
One might consider different spaces $V_{h,k}$, $V_{h,k'}$ for the finite element functions $\phi_h$ and $\psi_h$, respectively. However, both analysis and numerical experiments show that there is no significant advantage of taking $k \neq k'$. Similarly, one could use different scaling of the stabilization terms $s_h(\cdot,\cdot)$ in the two equations in  \eqref{NotesAR_10}, but this also turns out to be not significant.

For $\Gamma_h$ one can take a piecewise linear approximation as in \eqref{eqG}. For a higher order accuracy the parametric trace finite element method, briefly discussed in Remark~\ref{RemGammah} can be used. 

\section{Trace finite element method for velocity and pressure reconstruction} \label{Sectdiscreconstr}
In this section we introduce canonical discrete versions of the variational problems \eqref{Variationsformulierung_u} and \eqref{Variationsformulierung_p} for the approximate reconstruction of $\bu$ and $p$, given the discrete solution $\psi_h$ of \eqref{NotesAR_10}.
We first consider the discrete version of \eqref{Variationsformulierung_u}. For this we introduce the notation:
\begin{align}
 m_h(\bu,\bv)&:= \int_{\Gamma_h} \bu \cdot \bv \, ds_h,  \label{defmh} \\
 s_h(\bu,\bv)&:= \rho_u \int_{ \Omega^\Gamma_h} \left( \nabla \bu \bn_h \right) \cdot \left( \nabla \bu \bn_h\right) \, dx. \label{defsh}
\end{align}
For the parameter $\rho_u$ in the volume normal derivative stabilization of the velocity we restrict to the range
\begin{equation} \label{choicerho}
   \rho_u \sim h.
\end{equation}
The scaling $\rho_u \sim h$ is motivated by the error analysis in Section~\ref{SectPressureVelocity}, cf. Remark~\ref{rem_approxU}.
The approximate reconstruction of the tangential velocity $\bu^\ast$  is given by the unique solution $\bu_h$ of the following problem:
Determine $\bu_h \in \left( V_{h,k_u} \right)^3$ such that
\begin{align}
\label{Variationsformulierung_u_diskret}
m_h(\bu_h , \bv_h) + s_h(\bu_h,\bv_h) = \int_{\Gamma_h} (\tilde{\bn}_h \times \nablaGh \psi_h)\cdot \bv_h \,ds_h \quad \forall~\bv_h \in \left( V_{h,k_u} \right)^3,
\end{align}
 with an approximate  normal $\tilde{\bn}_h$. The reason why we introduce yet another normal approximation $\tilde{\bn}_h $, besides $\bn_h$, is the following.
The use of  the discrete normal $\bn_h$ in the right-hand side of \eqref{Variationsformulierung_u_diskret} leads to suboptimal discretization error bounds. We propose to use a normal approximation $\tilde{\bn}_h$ that is one order more accurate than $\bn_h$, cf. Remark~\ref{rem_connection_up}. A specific  choice for $\tilde{\bn}_h$ will be discussed in Section~\ref{SectExp}.

We now explain the reconstruction of the pressure solution $p^\ast$. For this we define
\begin{align*}
s_h(p_h,\xi_h) & := \rho_p \int_{ \Omega^\Gamma_h} \left( \bn_h \cdot \nabla p_h \right) \left( \bn_h \cdot \nabla \xi_h \right) \, dx.
\end{align*}
Note that for this stabilization bilinear form  $s_h(\cdot,\cdot)$ we use the same notation as in \eqref{defsh}. For the stability parameter $\rho_p$ in this pressure stabilization term we restrict to the same range as in \eqref{rhoscaling}: 
 \begin{equation} \label{stabrho}  h \lesssim \rho_p \lesssim h^{-1}.
 \end{equation}
The discrete variational formulation of (\ref{Variationsformulierung_p}) is as follows: Determine $p_h \in V_{h,k_p}$ with $\int_{\Gamma_h} p_h \, ds_h=0$,  such that
\begin{align}
\label{Variationsformulierung_p_diskret}
b_h(p_h , \xi_h) + s_h(p_h,\xi_h) = \int_{\Gamma_h} (K_h \RotGh \psi_h + \blf_h )\cdot \nablaGh \xi_h \,ds_h \quad \forall~\xi_h \in V_{h,k_p}.
\end{align}
 \begin{remark} \rm
We comment on certain important properties of the overall discretization method. The velocity solution $\bu$ that solves  \eqref{Variationsformulierung_u} is by construction \emph{tangential to $\Gamma$ and  divergence-free}. Hence, up to a discretization error, the discrete solution $\bu_h$ of \eqref{Variationsformulierung_u_diskret} also has these properties. Therefore we do not need a Lagrange multiplier or a penalty approach to enforce these two crucial properties of the velocity. The problem \eqref{NotesAR_8} for the stream function $\psi$ consists of two coupled \emph{scalar} second order surface partial differential equations. For such problems well-established techniques, for example the surface finite element method \cite{DEreview}, are available. The method presented above has a straightforward extension to time-dependent Stokes equations (on a stationary surface); in such a setting one can use the stream function $\psi$ to follow the dynamics of the problem  and the reconstruction of $\bu$ and $p$ 
can be performed only when needed. Two disadvantages of our method are the following. Firstly, it can be applied only to a simply connected surface. If $\Gamma$ does not have this property, there are nonzero harmonic velocity fields, which are not a-priori known and difficult to determine. Secondly,   the reconstruction of $\bu$ is based on a differentiation of the stream function $\psi$ and thus we will lose one order of accuracy when computing a reconstruction based on  \eqref{Variationsformulierung_u_diskret}.  From the error analysis and numerical experiments presented below  we see that for optimal order discretization error bounds  for polynomials of degree $k_u$ in the velocity reconstruction,  the finite elements used in the discretization of the stream function problem must be of degree at least $k_u+1$, cf. Remark~\ref{rem_connection_up}. 
\end{remark}
\ \\

In the sections below we present  a  discretization error analysis of the methods \eqref{NotesAR_10}, \eqref{Variationsformulierung_u_diskret} and \eqref{Variationsformulierung_p_diskret}. In this analysis we make the simplifying assumption that there are  \emph{no geometry errors}, i.e., $\Gamma_h=\Gamma$. We comment on this in Remark~\ref{commentassumption}.

\section{Analysis of the stream function formulation} \label{SectAnalysisStreamfunction}
In this section we present an error analysis of the discretization \eqref{NotesAR_10} for the simplified case $\Gamma_h=\Gamma$.  This means that in the discrete problem  in \eqref{NotesAR_10} we use the bilinear forms as in \eqref{defm}--\eqref{defbK}, $s_h(\psi_h,\eta_h):= \rho \int_{\Gamma} (\bn \cdot \nabla \psi_h)(\bn\cdot \nabla \eta_h)\, ds$ and a right-hand side functional as in \eqref{rhsg}. Our analysis uses a technique inspired by the paper \cite{Falk1980}. In that paper a general framework for the analysis of mixed problems is presented which applies to the discretization of the biharmonic equation (reformulated as a coupled second order system) in Euclidean space. The discrete problem \eqref{NotesAR_10} that we consider, however, does \emph{not} fit into the framework presented in \cite{Falk1980}. This is caused by the bilinear form $b_{K}(\xi,\eta)=2 \int_{\Gamma} (1-K) \nablaG \xi \cdot \nablaG  \eta \, ds$. This term does not occur in the framework presented in \cite{
Falk1980}, which involves only the terms $m(\cdot,\cdot)$ (denoted by $a(\cdot,\cdot)$ in \cite{Falk1980}) and $b(\cdot,\cdot)$. Similar to $b(\cdot,\cdot)$ the ``new'' term $b_{K}(\xi,\eta)$ contains gradients of its arguments, but opposite to $b(\cdot,\cdot)$ it does \emph{not} have an ellipticity property. This is due to the fact that for general smooth closed surfaces $\Gamma$ we do not have the bound  $K(x) < 1 $ for all $x \in \Gamma$ (or $K(x) < 0 $, for the case $\alpha=0$ in \eqref{defblfa}).  To be able to control the bilinear form $b_{K}(\cdot,\cdot)$, in the error analysis we use the fundamental estimate \eqref{NotesAR_7} that is derived in the next section, cf. Remark~\ref{rem_approxU}.

The structure of the analysis is as follows. In Section~\ref{sectprelim} we collect a few relevant results known from the literature and derive the fundamental inequality \eqref{NotesAR_7}. In Section~\ref{secterrstreamfunction} discretization error bounds for the solution of the discrete stream function problem \eqref{NotesAR_10} are derived. For this we first prove  bounds for $\phi^\ast-\phi_h^\ast$ in different norms (Section~\ref{secterrA}), then we present estimates for $\psi^\ast - \psi_h^\ast$ (Section~\ref{secterrB}), and combining these results we obtain discretization error bounds (Section~\ref{secterrC}). 

\subsection{Preliminaries} \label{sectprelim}
The following Korn type inequality is derived in \cite{Reusken_Streamfunction}:  there exists $c_K>0$:
\begin{equation} \label{Korn}
\|\bv\|_{L^2(\Gamma)} +  \|E_s(\bv) \|_{L^2(\Gamma)} \geq c_K \| \bv \|_{H^1(\Gamma)} \quad \text{for all}~\bv \in \bH^1_t(\Gamma).
\end{equation}
From Lemma \ref{lemma2} it follows that there are  strictly positive constants $c_0,\hat{c}_0$ such that 
\begin{align}
\label{NotesAR_2}
c_0 \| \psi \|_{H^2(\Gamma)} \leq \| \RotG \psi \|_{H^1(\Gamma)} \leq \hat{c}_0 \| \psi \|_{H^2(\Gamma)} \quad \text{for all}~\psi \in H^2_\ast(\Gamma).
\end{align}
 In the next lemma we present a fundamental result for the bilinear form $b_K(\cdot,\cdot)$ that is derived using its connection to the surface Stokes problem. 
\begin{lemma}
\label{lem_NotesAR_1}
The following  inequality holds:
\begin{align*}
\| \DeltaG \psi \|_{L^2(\Gamma)}^2 + b_K(\psi,\psi) \geq c_F \| \DeltaG \psi \|_{L^2(\Gamma)}^2 \quad \text{for all}~ \psi \in H^2_*(\Gamma),
\end{align*}
with $c_F:=2 c_K c_0^2 > 0$, and $c_K$, $c_0$ as in \eqref{Korn} and \eqref{NotesAR_2}, respectively.
\end{lemma}
\begin{proof}
Let $\psi \in H^2_*(\Gamma)$ be given and define $\bu := \RotG \psi \in \bH^{1}_{t, \rm div}(\Gamma)$. Using \eqref{identa} (with $\alpha=1$), \eqref{Korn} and \eqref{NotesAR_2} we obtain
\begin{align*}
\| \DeltaG \psi \|_{L^2(\Gamma)}^2 + b_K(\psi,\psi) = 2  a_1(\bu,\bu) \geq 2 c_K \| \RotG \psi \|_{H^1(\Gamma)}^2 \geq 2 c_K c_0^2 \| \psi \|_{H^2(\Gamma)}^2.
\end{align*}
\end{proof}
\ \\[1ex]
This yields the following corollary.
\begin{corollary}
\label{cor_NotesAR_1}
Consider the standard Laplace-Beltrami equation: Given $f \in L^2_\ast(\Gamma):=\{\, f \in L^2(\Gamma)~|~\int_\Gamma f\, ds =0\,\}$, determine $\psi \in H^1_\ast (\Gamma)$ such that
\begin{align}
\label{NotesAR_4}
b(\psi,\xi) = \int_{\Gamma} f \xi \, ds \quad \text{for all}~ \xi \in  H^1 (\Gamma).
\end{align}
The unique solution $\psi$ has regularity $\psi \in H^2(\Gamma)$ and 
\begin{align*}
\| f \|_{L^2(\Gamma)}^2 + b_K(\psi,\psi) \geq c_F \| f \|_{L^2(\Gamma)}^2
\end{align*}
holds.
\end{corollary}
We derive a discrete variant of Corollary \ref{cor_NotesAR_1}. 
\begin{corollary}
\label{cor_NotesAR_2} 
For $f \in L^2_\ast(\Gamma)$ consider the following discrete Laplace-Beltrami problem (with stabilization): Determine $\psi_h \in V_{h,k}$ with $\int_\Gamma \psi_h \, ds=0$  such that
\begin{align}
\label{NotesAR_5}
b(\psi_h,\xi_h) + s_h(\psi_h,\xi_h) = \int_{\Gamma} f \xi_h \, ds \quad \text{for all}~ \xi_h \in  V_{h,k}.
\end{align}
Let $\psi_h$ be the unique solution of \eqref{NotesAR_5}.
  For $h$ sufficiently small the following estimate holds:
\begin{align*}
\| f \|_{L^2(\Gamma)}^2 + b_K(\psi_h,\psi_h) \geq \frac{1}{2} c_F \| f \|_{L^2(\Gamma)}^2.
\end{align*}
\end{corollary}
\begin{proof} 
Let $\psi$ and $\psi_h$ be the solution of \eqref{NotesAR_4} and \eqref{NotesAR_5}, respectively.
From the literature (Theorem 5.6 in \cite{GrandeLehrenfeldReusken}) we have the stability and error estimates:
\begin{align*}
  \big(b(\psi_h,\psi_h) +s_h(\psi_h,\psi_h)\big)^\frac12  & \leq c \|f\|_{L^2(\Gamma)}, \\
  \|\nablaG(\psi - \psi_h )\|_{L^2(\Gamma)} & \leq c h \|f\|_{L^2(\Gamma)}.
\end{align*}
Using these we obtain
\begin{align*}
| b_K(\psi_h,\psi_h) - b_K(\psi,\psi) | &= | b_K(\psi_h - \psi , \psi_h) + b_K(\psi,\psi_h - \psi) | \\
&\leq c \|\nablaG(\psi-\psi_h)\|_{L^2(\Gamma)} \|\nablaG\psi_h\|_{L^2(\Gamma)}  
\leq c h  \| f \|_{L^2(\Gamma)}^2.
\end{align*}
Combining this with the result in Corollary \ref{cor_NotesAR_1} yields
\begin{align*}
\| f \|_{L^2(\Gamma)}^2 + b_K(\psi_h,\psi_h) &= \| f \|_{L^2(\Gamma)}^2 + b_K(\psi,\psi) + \left( b_K(\psi_h,\psi_h) - b_K(\psi,\psi)  \right) \\
&\geq c_F \| f \|_{L^2(\Gamma)}^2 - c h  \| f \|_{L^2(\Gamma)}^2 
\geq \frac{1}{2} c_F \| f \|_{L^2(\Gamma)}^2,
\end{align*}
for $h$ sufficiently small.
\end{proof}
\ \\[1ex]
As a direct consequence we have  the following estimate for the solution $\psi_h$ of \eqref{NotesAR_5}, which plays a key role in the analysis below:
\begin{align}
\label{NotesAR_7}
- b_K(\psi_h,\psi_h) \leq (1-\frac{1}{2} c_F) \| f \|_{L^2(\Gamma)}^2.
\end{align}

We recall a result that is standard in the analysis of trace finite element methods \cite{GrandeLehrenfeldReusken,burmanembedded}:
\begin{align}
\label{lem_NotesAR_3_help}
\| \xi_h \|_{L^2(\Omega^\Gamma_h)}^2 \lesssim h \| \xi_h \|_{L^2(\Gamma)}^2 + h^2 \|\bn\cdot \nabla \xi_h\|_{L^2(\Omega_h^\Gamma)}^2 \quad \text{for all} ~ \xi_h \in V_{h,k}.
\end{align}

\begin{lemma}
\label{lem_NotesAR_3}
The problem \eqref{NotesAR_10} has a unique solution.
\end{lemma}
\begin{proof}
Take $g \equiv 0$. We have to show $\psi_h=\phi_h=0$. For the choice $\xi_h=\psi_h$ and $\eta_h=\phi_h$ as test functions in \eqref{NotesAR_10} we get by subtracting both equations
\begin{align}
\label{NotesAR_11}
m(\phi_h,\phi_h) + b_K(\psi_h,\psi_h) = 0.
\end{align} 
Note that the first equation in \eqref{NotesAR_10} is of the form as in \eqref{NotesAR_5} with $f=-\phi_h$. Using  Corollary \ref{cor_NotesAR_2} yields
\begin{align*}
\| \phi_h \|_{L^2(\Gamma)}^2 + b_K(\psi_h,\psi_h) \geq \frac{1}{2} c_F \| \phi_h \|_{L^2(\Gamma)}^2,
\end{align*}
and with \eqref{NotesAR_11} we obtain
\begin{align*}
0 \geq \frac{1}{2} c_F \| \phi_h \|_{L^2(\Gamma)}^2, \quad \text{i.e.,}~\| \phi_h \|_{L^2(\Gamma)} = 0.
\end{align*}
 Testing  the first equation of \eqref{NotesAR_10} with $\eta_h=\psi_h$, applying the Cauchy-Schwarz inequality in combination with the previous result we get
\begin{align*}
b(\psi_h,\psi_h) + s_h(\psi_h, \psi_h) = - m(\phi_h,\psi_h) \leq \|\phi_h\|_{L^2(\Gamma)} \|\psi_h\|_{L^2(\Gamma)} = 0.
\end{align*}
 Using a Poincare inequality, $\|\psi_h \|_{L^2(\Gamma)}^2 \leq c\, b(\psi_h,\psi_h)$, and the inequality \eqref{lem_NotesAR_3_help}, this implies 
$\|\psi_h\|_{L^2(\Omega^\Gamma_h)}=0$ and therefore, $\psi_h=0$. Using $\xi_h=\phi_h$ as a test function in the second equation in \eqref{NotesAR_10} and the Cauchy-Schwarz inequality we get
\begin{align*}
 b(\phi_h,\phi_h) + s_h(\phi_h, \phi_h) = b_K(\psi_h,\phi_h)  = 0.
\end{align*}
With the same arguments as above we  conclude $\phi_h=0$. 
\end{proof}
\ \\[1ex]
In the following we denote the unique solution by $\psi_h^*$, $\phi_h^*$.
 
\subsection{Error analysis for the stream function} \label{secterrstreamfunction}
It is convenient to introduce the notation
\[
  A(\psi,\xi):= b(\psi,\xi)+s_h(\psi,\xi).
\]
The corresponding seminorm is denoted by $\|\cdot\|_A$. The usual $H^1(\Gamma)$ semi-norm is defined by $|\cdot|_1:=b(\cdot,\cdot)^\frac12$. For functions $\psi  \in H^1(\Gamma)$ we always use a constant extension  along normals, which is also denoted by $\psi$.  Hence, $s_h(\psi,\xi_h)=0$ for all $\xi_h \in V_{h,k}$ holds. 
We introduce the projection $\Pi_h \, : \, H^1(\Gamma) \to V_{h,k}\cap H_\ast^1(\Gamma)$ defined by
\begin{align}
\label{NotesAR_Projection}
A(\Pi_h \psi , \eta_h)  = A(\psi,\eta_h)=b(\psi , \eta_h) \quad \text{for all} ~ \eta_h \in V_{h,k},
\end{align}
i.e.,
\[
  b(\Pi_h \psi- \psi, \eta_h)+ s_h(\Pi_h \psi- \psi,\eta_h)=0 \quad \text{for all} ~ \eta_h \in V_{h,k}.
\]
This projection corresponds to the solution operator of the discretization of the scalar Laplace-Beltrami equation in \eqref{NotesAR_5}.  For this problem, discretization error analyses are available in the literature \cite{GrandeLehrenfeldReusken,reusken2015analysis,burmanembedded,Bonito2019}, which result in \emph{optimal discretization error bounds}, both in the energy norm and the $L^2(\Gamma)$ norm. These results yield the following proposition:
\begin{proposition} \label{propestimates} Let $m \geq 3$ be such that the solution of \eqref{NotesAR_8} has regularity $\psi^* \in H^m(\Gamma)$ and $\phi^* \in H^{m-2}(\Gamma)$. For the Laplace-Beltrami Galerkin projection $\Pi_h$ the following estimates hold:
\begin{align}
\label{NotesAR_13a}
\| \psi^* -\Pi_h \psi^* \|_A &\lesssim h^{r-1} \| \psi^* \|_{H^r(\Gamma)}, \quad 1 \leq r \leq s_1, \\
\label{NotesAR_13b}
\| \psi^* -\Pi_h \psi^* \|_{L^2(\Gamma)} &\lesssim h^{r} \| \psi^* \|_{H^r(\Gamma)} \quad 0 \leq r \leq s_1,\\
\label{NotesAR_14a}
\| \phi^* -\Pi_h \phi^* \|_A &\lesssim h^{r-1} \| \phi^* \|_{H^r(\Gamma)} \quad 1 \leq r \leq s_2, \\
\label{NotesAR_14b}
\| \phi^* -\Pi_h \phi^* \|_{L^2(\Gamma)} &\lesssim h^{r} \| \phi^* \|_{H^r(\Gamma)} \quad 0 \leq r \leq s_2,\\
\text{\rm with} ~ s_1:=\min \{m, k+1\}~&~\text{\rm and}~~ s_2:=\min \{m-2, k+1\}.\label{defs1s2}
\end{align}
\end{proposition}
\begin{remark} \label{RemDE} \rm 
The estimates in Proposition~\ref{propestimates} combined with \eqref{NotesAR_7} are the essential ingredients for the analysis below to work. Therefore, the analysis  also applies to the surface finite element method, for which the results \eqref{NotesAR_13a}--\eqref{NotesAR_14b}  are known to hold.
\end{remark}

\subsubsection{Bounds for the error $\phi^* - \phi_h^*$} \label{secterrA}
The analysis below is  along the same lines as in \cite{Falk1980}. However, as indicated already above, we have  an additional term $b_K(\cdot,\cdot)$ that has to be controlled.
\begin{theorem}
\label{NotesAR_Theorem1} Let $(\psi^\ast,\phi^\ast)$ and $(\psi_h^\ast,\phi_h^\ast)$ be the solutions of \eqref{NotesAR_8} and \eqref{NotesAR_10} {\rm(}with $\Gamma_h=\Gamma${\rm)}, respectively.
The following estimates hold:
\begin{align*}
\| \phi^* - \phi_h^* \|_{L^2(\Gamma)} &\lesssim h^{s_2} \| \phi^* \|_{H^{s_2}(\Gamma)} + h^{s_1-1} \| \psi^* \|_{H^{s_1}(\Gamma)} + | \psi^* - \psi_h^* |_1,\\
\| \phi^* - \phi_h^* \|_A &\lesssim h^{s_2-1} \| \phi^* \|_{H^{s_2}(\Gamma)} + | \psi^* - \psi_h^* |_1,
\end{align*}
with $s_1,s_2$ as in \eqref{defs1s2}.
\end{theorem}
\begin{proof}
It is convenient to introduce a notation for  Galerkin projection errors:
\begin{equation} \label{Galnotation}
  e_\psi:= \psi^\ast-\Pi_h\psi^\ast, ~~~ e_\phi:= \phi^\ast-\Pi_h\phi^\ast.
\end{equation}
From \eqref{NotesAR_8} and \eqref{NotesAR_10}  we obtain the following Galerkin relations:
\begin{align}
\label{NotesAR_15}
m(\phi^*-\phi_h^*,\eta_h) + b(\psi^*-\psi_h^*, \eta_h) - s_h(\psi_h^*, \eta_h) &= 0 \qquad \forall \, \eta_h \in V_{h,k},\\
\label{NotesAR_16}
b(\phi^*-\phi_h^*,\xi_h) - b_K(\psi^*-\psi_h^*, \xi_h) - s_h(\phi_h^*, \xi_h) &= 0 \qquad \forall \, \xi_h \in V_{h,k}.
\end{align}
Using the first identity \eqref{NotesAR_15} we get
\begin{align*}
m(\Pi_h \phi^*-\phi_h^*,\eta_h) &= m(\Pi_h \phi^*-\phi^*,\eta_h) + m(\phi^*-\phi_h^*,\eta_h) \\
&= m(\Pi_h \phi^*-\phi^*,\eta_h) + b(\psi_h^*-\psi^*, \eta_h) + s_h(\psi_h^*, \eta_h) \qquad \forall \, \eta_h \in V_{h,k}.
\end{align*}
Taking $\eta_h=\Pi_h \phi^*-\phi_h^*$ leads to
\begin{align}
\label{NotesAR_17}
\| \Pi_h \phi^*-\phi_h^* \|_{L^2(\Gamma)}^2 =& - m(e_\phi,\Pi_h \phi^*-\phi_h^*) \\
&+ \underbrace{b(\psi_h^*-\psi^*, \Pi_h \phi^*-\phi_h^*) + s_h(\psi_h^*, \Pi_h \phi^*-\phi_h^*)}_{=:(I)}.
\end{align}
We first consider (I). Using the projection property \eqref{NotesAR_Projection} and the second Galerkin relation \eqref{NotesAR_16} yields
\begin{align}
\notag
(I) &\overset{\eqref{NotesAR_Projection}}{=} b(\psi_h^*- \Pi_h \psi^*, \Pi_h \phi^*-\phi_h^*) - s_h(\Pi_h \psi^*, \Pi_h \phi^*-\phi_h^*) + s_h(\psi_h^*, \Pi_h \phi^*-\phi_h^*) \\
\notag
& \hspace*{2mm}=\hspace*{2mm} b( \Pi_h \phi^*-\phi_h^* , \psi_h^*- \Pi_h \psi^* ) + s_h( \Pi_h \phi^*-\phi_h^*, \psi_h^* - \Pi_h \psi^*) \\
\notag
&\overset{\eqref{NotesAR_Projection}}{=} b( \phi^*-\phi_h^* , \psi_h^*- \Pi_h \psi^* ) - s_h( \phi_h^*, \psi_h^* - \Pi_h \psi^*) \\
\notag
&\overset{\eqref{NotesAR_16}}{=} b_K(\psi^*-\psi_h^*, \psi_h^* - \Pi_h \psi^*) \\
\label{NotesAR_18}
& \hspace*{2mm}=\hspace*{2mm} b_K(e_\psi, \psi_h^* - \Pi_h \psi^*) - b_K(\psi_h^* - \Pi_h \psi^*, \psi_h^* - \Pi_h \psi^*).
\end{align}
Due to the Galerkin relation \eqref{NotesAR_15} and the projection property \eqref{NotesAR_Projection} we get
\begin{align*}
m(\phi^*-\phi_h^*,\eta_h) & = b(\psi_h^* - \psi^*, \eta_h) + s_h(\psi_h^*, \eta_h) \\
& = b(\psi_h^* - \Pi_h \psi^*, \eta_h) + s_h(\psi_h^* - \Pi_h \psi^*, \eta_h)
\end{align*}
for all $\eta_h \in V_{h,k}$. Hence, $\psi_h^* - \Pi_h \psi^*$ is the  solution of the discrete Laplace-Beltrami problem \eqref{NotesAR_5} with right-hand side $f=\phi^*-\phi_h^*$. Using \eqref{NotesAR_7} we obtain
\begin{align}
\label{NotesAR_18_help}
- b_K(\psi_h^* - \Pi_h \psi^* , \psi_h^* - \Pi_h \psi^*) \leq (1-\frac{1}{2} c_F) \| \phi^*-\phi_h^* \|^2_{L^2(\Gamma)}.
\end{align}
Combining the results \eqref{NotesAR_17} and \eqref{NotesAR_18} and using the Cauchy-Schwarz inequality and \eqref{NotesAR_18_help}, we obtain
\begin{equation} \label{eqcentral} \begin{split}
\| \Pi_h \phi^*-\phi_h^* \|_{L^2(\Gamma)}^2 =& - m(e_\phi,\Pi_h \phi^*-\phi_h^*) + b_K(e_\psi , \psi_h^* - \Pi_h \psi^*) \\
&- b_K(\psi_h^* - \Pi_h \psi^*, \psi_h^* - \Pi_h \psi^*) \\
\leq & \, \| e_\phi \|_{L^2(\Gamma)} \| \Pi_h \phi^*-\phi_h^* \|_{L^2(\Gamma)} + (1-\frac{1}{2} c_F) \| \phi^*-\phi_h^* \|^2_{L^2(\Gamma)} \\
&+ c | e_\psi|_1 | \psi_h^* - \Pi_h \psi^*|_1 \\
\leq & \, \frac{1}{2\beta} \| e_\phi\|_{L^2(\Gamma)}^2 + \frac{\beta}{2} \| \Pi_h \phi^*-\phi_h^* \|_{L^2(\Gamma)}^2 \\
&+ (1-\frac{1}{2} c_F) (1+\frac{1}{\alpha}) \| e_\phi \|^2_{L^2(\Gamma)} \\
&+ (1-\frac{1}{2} c_F) (1+\alpha) \| \Pi_h \phi^*-\phi_h^* \|^2_{L^2(\Gamma)} \\
&+ c | e_\psi|_1 | \psi_h^* - \Pi_h \psi^*|_1
\end{split}
\end{equation}
for all $\alpha, \beta > 0$ and a suitable constant $c$. We take $\alpha$ and $\beta$ such that $\frac{\beta}{2} + (1-\frac{1}{2}c_F)(1+\alpha) <1$ and then shift the term $\| \Pi_h \phi^*-\phi_h^* \|_{L^2(\Gamma)}^2$ in \eqref{eqcentral} to the left-hand side. Applying the triangle inequality $| \psi_h^* - \Pi_h \psi^*|_1 \leq  | \psi_h^* - \psi^*|_1 +  | e_\psi|_1$ yields (for $h$ sufficiently small)
\begin{align*}
\| \Pi_h \phi^*-\phi_h^* \|_{L^2(\Gamma)}^2 \lesssim & \, \| e_\phi \|_{L^2(\Gamma)}^2 + | e_\psi|_1^2 + | e_\psi|_1 | \psi^* - \psi_h^*|_1 \\
  \lesssim & \| e_\phi \|_{L^2(\Gamma)}^2 + | e_\psi|_1^2 +  | \psi^* - \psi_h^*|_1^2.
\end{align*}
With the projection error bounds \eqref{NotesAR_13a} and \eqref{NotesAR_14b} we get
\begin{align*}
\| \Pi_h \phi^*-\phi_h^* \|_{L^2(\Gamma)} \lesssim
& \, h^{s_2} \| \phi^* \|_{H^{s_2}(\Gamma)} + h^{s_1-1} \| \psi^* \|_{H^{s_1}(\Gamma)} + | \psi^* - \psi_h^*|_1.
\end{align*}
Combining this with $\| \phi^* - \phi_h^* \|_{L^2(\Gamma)} \leq \| e_\phi\|_{L^2(\Gamma)} + \| \Pi_h \phi^* - \phi_h^* \|_{L^2(\Gamma)} $ and the projection error bound \eqref{NotesAR_14b} we obtain the  bound for the error $\| \phi^* - \phi_h^* \|_{L^2(\Gamma)}$ .

For deriving the bound for $\| \phi^* - \phi_h^* \|_A$ we start with the second Galerkin relation \eqref{NotesAR_16} and use the projection property \eqref{NotesAR_Projection}, which yields
\[
 b(\phi^*-\phi_h^*,\xi_h) - s_h(\phi_h^*, \xi_h) = b_K(\psi^*-\psi_h^*, \xi_h),
\]
hence, 
\[ b(\Pi_h \phi^*-\phi_h^*,\xi_h) + s_h(\Pi_h \phi^* - \phi_h^*, \xi_h) = b_K(\psi^*-\psi_h^*, \xi_h) \quad \text{for all}~\xi_h \in V_{h,k}.
\]
Taking $\xi_h=\Pi_h \phi^*-\phi_h^*$ yields
\begin{align*}
\| \Pi_h \phi^* - \phi_h^* \|_A^2 = b_K(\psi^*-\psi_h^*, \Pi_h \phi^*-\phi_h^*) \lesssim |\psi^*-\psi_h^*|_1 |\Pi_h \phi^*-\phi_h^*|_1.
\end{align*}
Using $|\Pi_h \phi^*-\phi_h^*|_1 \leq \|\Pi_h \phi^*-\phi_h^*\|_A$ we conclude 
\begin{align*}
\| \Pi_h \phi^* - \phi_h^* \|_A \lesssim |\psi^*-\psi_h^*|_1.
\end{align*}
Combining this with $\| \phi^* - \phi_h^* \|_A \leq \| e_\phi \|_A + \| \Pi_h \phi^* - \phi_h^* \|_A$ and the projection error bound \eqref{NotesAR_14a} leads to the desired error bound.
\end{proof}
\ \\[1ex]
\begin{remark} \rm 
In the  proof above it is essential that for $\gamma:=\frac{\beta}{2} + (1-\frac{1}{2}c_F)(1+\alpha)$ we have the bound $\gamma < 1$, cf. \eqref{eqcentral}. 
For this to hold it is essential that in the estimate \eqref{NotesAR_7} we have a constant $1-\frac{1}{2} c_F <1$.
\end{remark}
\ \\[1ex]

\subsubsection{Bound for the error $\psi^* - \psi_h^*$} \label{secterrB}
In the next theorem an error bound for the error $| \psi^* - \psi_h^* |_1$ is derived.
\begin{theorem}
\label{NotesAR_Theorem2}  Let $(\psi^\ast,\phi^\ast)$ and $(\psi_h^\ast,\phi_h^\ast)$ be the solutions of \eqref{NotesAR_8} and \eqref{NotesAR_10} {\rm(}with $\Gamma_h=\Gamma${\rm)}, respectively.
The following estimate holds:
\begin{align*}
| \psi^* - \psi_h^* |_1 &\lesssim h^{s_1-1} \| \psi^* \|_{H^{s_1}(\Gamma)} + h \| \phi^* - \phi_h^* \|_{L^2(\Gamma)} + h^{\min \{ 2, k \} } \| \phi^* - \phi_h^* \|_A,
\end{align*}
with $s_1$ as in \eqref{defs1s2}.
\end{theorem}
\begin{proof}
Take $g(\xi)= -2 \int_{\Gamma} \blf \cdot \RotG \xi \, ds$ with  $\blf:=-\frac{1}{2} \RotG \left( \psi^* - \psi_h^* \right)$ in the problem \eqref{NotesAR_8}. The corresponding unique solution, denoted by $\hat \psi$, $\hat \phi$, satisfies
\begin{align}
\label{NotesAR_19b}
m(\hat \phi,\eta) + b(\hat \psi, \eta) &= 0 \quad \text{for all}~ \eta \in H^1(\Gamma),\\
\label{NotesAR_19c}
b(\hat \phi,\xi) - b_K(\hat \psi, \xi) &= \int_{\Gamma} \RotG \left( \psi^* - \psi_h^* \right) \cdot \RotG \xi \, ds. \quad \text{for all} ~ \xi \in H^1(\Gamma),
\end{align}
and the regularity estimates
\begin{align}
\label{NotesAR_20a}
\| \hat \psi \|_{H^{3}(\Gamma)} \lesssim \| \blf \|_{L^2(\Gamma)} \lesssim | \psi^* - \psi_h^* |_1, \\
\label{NotesAR_20b}
\| \hat \phi \|_{H^{1}(\Gamma)} \lesssim \| \blf \|_{L^2(\Gamma)} \lesssim | \psi^* - \psi_h^* |_1.
\end{align}
Again the solutions $\hat \psi$ and $\hat \phi$ are extended constantly along normals. We use notation as in \eqref{Galnotation} and introduce for a better readability the following Galerkin projection errors:
\begin{equation*}
  \hat e_\psi:= \hat \psi - \Pi_h \hat \psi, ~~~ \hat e_\phi:= \hat \phi - \Pi_h \hat \phi.
\end{equation*}
Taking $\xi=\psi^* - \psi_h^* \in H^1(\Gamma)$ in \eqref{NotesAR_19c} yields
\begin{align}
\label{NotesAR_21}
|\psi^* - \psi_h^*|_1^2 = b(\hat \phi,\psi^* - \psi_h^*) - b_K(\hat \psi, \psi^* - \psi_h^*).
\end{align}
We rewrite the first term on the right-hand side of \eqref{NotesAR_21} with the help of the Galerkin relations \eqref{NotesAR_15}, \eqref{NotesAR_16}, the first equation of \eqref{NotesAR_19b} and the projection property \eqref{NotesAR_Projection}:
\begin{align*}
b(\hat \phi,\psi^* - \psi_h^*) & \hspace*{2mm}=\hspace*{2mm} b(\hat e_\phi ,\psi^* - \psi_h^*) + b(\Pi_h \hat \phi,\psi^* - \psi_h^*) \\
& \hspace*{2mm}=\hspace*{2mm} b(\hat e_\phi , e_\psi) + b(\hat e_\phi , \Pi_h \psi^* - \psi_h^*) + b(\Pi_h \hat \phi,\psi^* - \psi_h^*) \\
&\overset{\eqref{NotesAR_Projection}}{=} b(\hat e_\phi , e_\psi) + s_h( \Pi_h \hat \phi , \Pi_h \psi^* - \psi_h^*) + b(\Pi_h \hat \phi,\psi^* - \psi_h^*) \\
&\overset{\eqref{NotesAR_15}}{=} b(\hat e_\phi , e_\psi) + s_h( \Pi_h \hat \phi , \Pi_h \psi^* - \psi_h^*) +s_h(\Pi_h \hat \phi, \psi_h^*) \\
& \hspace*{8mm} - m(\phi^* -\phi_h^*,\Pi_h \hat \phi) \\
& \hspace*{2mm}=\hspace*{2mm} b(\hat e_\phi , e_\psi) + s_h( \Pi_h \hat \phi , \Pi_h \psi^*) - m(\phi_h^* -\phi^*,\hat e_\phi) \\
& \hspace*{8mm} + m(\phi_h^* -\phi^*, \hat \phi) \\
&\overset{\eqref{NotesAR_19b}}{=}  b(\hat e_\phi , e_\psi) + s_h( \Pi_h \hat \phi , \Pi_h \psi^*) - m(\phi_h^* -\phi^*, \hat e_\phi) \\
& \hspace*{8mm} + b(\phi^* -\phi_h^*, \hat \psi) \\
& \hspace*{2mm}=\hspace*{2mm} b(\hat e_\phi , e_\psi) + s_h( \Pi_h \hat \phi , \Pi_h \psi^*) - m(\phi_h^* -\phi^*, \hat e_\phi) \\
& \hspace*{8mm} + b(\phi^* -\phi_h^*, \hat e_\psi) + b(\phi^* -\phi_h^*, \Pi_h \hat \psi) \\
&\overset{\eqref{NotesAR_16}}{=} b(\hat e_\phi , e_\psi) + s_h( \Pi_h \hat \phi , \Pi_h \psi^*) - m(\phi_h^* -\phi^*, \hat e_\phi ) \\
& \hspace*{8mm} + b(\phi^* -\phi_h^*, \hat e_\psi ) + b_K(\psi^* -\psi_h^*, \Pi_h \hat \psi) + s_h(\phi_h^*, \Pi_h \hat \psi).
\end{align*}
With this we can conclude from \eqref{NotesAR_21}
\begin{align*}
|\psi^* - \psi_h^*|_1^2 =& \, b(\hat e_\phi , e_\psi) + s_h( \Pi_h \hat \phi , \Pi_h \psi^*) - m(\phi_h^* -\phi^*, \hat e_\phi) + b(\phi^* -\phi_h^*, \hat e_\psi ) \\
& + s_h(\phi_h^*, \Pi_h \hat \psi) - b_K(\psi^* -\psi_h^*, \hat e_\psi ) \\
\leq & \, \| \hat e_\phi \|_A \| e_\psi \|_A + \| \phi_h^* -\phi^* \|_{L^2(\Gamma)} \| \hat e_\phi \|_{L^2(\Gamma)} + \| \phi^* -\phi_h^* \|_A \| \hat e_\psi \|_A \\
& + c | \psi^* -\psi_h^*|_1 | \hat e_\psi |_1.
\end{align*}
Using the projection error bounds \eqref{NotesAR_13a}, \eqref{NotesAR_14a} and \eqref{NotesAR_14b} in combination with the regularity estimates \eqref{NotesAR_20a} and \eqref{NotesAR_20b} yields
\begin{align*}
|\psi^* - \psi_h^*|_1^2 \lesssim & \, \| e_\psi \|_A | \psi^* -\psi_h^*|_1 + h \| \phi_h^* -\phi^* \|_{L^2(\Gamma)} | \psi^* -\psi_h^*|_1 \\
& + h^{\min \{2,k\} } \| \phi^* -\phi_h^* \|_A | \psi^* -\psi_h^*|_1 + h^{\min \{2,k\} } | \psi^* -\psi_h^*|_1^2.
\end{align*}
The last term can be shifted to the left-hand side since $h$ is sufficiently small and $k \geq 1$. Applying the projection error bound \eqref{NotesAR_13a} leads to the claimed estimate of the theorem.
\end{proof}
\ \\[1ex]

\subsubsection{Discretization error bounds} \label{secterrC}
Combining the results in Theorem \ref{NotesAR_Theorem1} and Theorem \ref{NotesAR_Theorem2} we obtain discretization error bounds.
\begin{theorem}
\label{thmDiscretizationErrorStreamfunction} Let $(\psi^\ast,\phi^\ast)$ and $(\psi_h^\ast,\phi_h^\ast)$ be the solutions of \eqref{NotesAR_8} and \eqref{NotesAR_10} {\rm(}with $\Gamma_h=\Gamma${\rm)}, respectively. Let $m \geq 3$ be such that  
 $\psi^* \in H^1_*(\Gamma) \cap H^m(\Gamma)$ and $\phi^* \in H^{m-2}(\Gamma)$.  With $s_1=\min \{m, k+1\}$, $s_2=\min \{m-2, k+1\}$ and $k^*:=0$ if $k=1$ and $k^*:=1$ if $k \geq 2$ the following error bounds hold:
\begin{align}
\label{NotesAR_22}
| \psi^* - \psi_h^* |_1 &\lesssim h^{s_1-1} \| \psi^* \|_{H^{s_1}(\Gamma)} + h^{s_2+k^*} \| \phi^* \|_{H^{s_2}(\Gamma)}, \\
\label{NotesAR_23b}
\| \phi^* - \phi_h^* \|_A &\lesssim h^{s_1-1} \| \psi^* \|_{H^{s_1}(\Gamma)} + h^{s_2-1} \| \phi^* \|_{H^{s_2}(\Gamma)}, \\
\label{NotesAR_23a}
\| \phi^* - \phi_h^* \|_{L^2(\Gamma)} &\lesssim  h^{s_1-1} \| \psi^* \|_{H^{s_1}(\Gamma)} + h^{s_2} \| \phi^* \|_{H^{s_2}(\Gamma)}.
\end{align}

In particular, we have the following result for $k \geq 2$ and $\psi^* \in H^{k+1}(\Gamma)$:
\begin{align}
\label{NotesAR_22_kgeq2}
| \psi^* - \psi_h^* |_1 &\lesssim h^{k} \| \psi^* \|_{H^{k+1}(\Gamma)} + h^{k} \| \phi^* \|_{H^{k-1}(\Gamma)} \lesssim h^{k} \| \psi^* \|_{H^{k+1}(\Gamma)}.
\end{align}

\end{theorem}
\begin{proof}
For the first result we start with the bound in Theorem \ref{NotesAR_Theorem2} and insert the estimates of Theorem \ref{NotesAR_Theorem1}:
\begin{align*}
| \psi^* - \psi_h^* |_1 \lesssim & \,  h^{s_1-1} \| \psi^* \|_{H^{s_1}(\Gamma)} + h \left( h^{s_2} \| \phi^* \|_{H^{s_2}(\Gamma)} + h^{s_1-1} \| \psi^* \|_{H^{s_1}(\Gamma)} + | \psi^* - \psi_h^* |_1 \right) \\
&+ h^{\min \{ 2, k \} } \left( h^{s_2-1} \| \phi^* \|_{H^{s_2}(\Gamma)} + | \psi^* - \psi_h^* |_1 \right) \\
\lesssim & \,  h^{s_1-1} \| \psi^* \|_{H^{s_1}(\Gamma)} + h^{s_2 + k^*} \| \phi^* \|_{H^{s_2}(\Gamma)} + h | \psi^* - \psi_h^* |_1.
\end{align*}
For $h$ sufficiently small we can shift the term $h | \psi^* - \psi_h^* |_1$ to the left-hand side, which leads to the desired error bound. 

The second and the third result are obtained by inserting the estimate \eqref{NotesAR_22} into the first and the second error bound of Theorem \ref{NotesAR_Theorem1}. The discretization error bound \eqref{NotesAR_22_kgeq2} follows  from \eqref{NotesAR_22} and $\|\phi^*\|_{H^{k-1}(\Gamma)} = \| \DeltaG \psi^*\|_{H^{k-1}(\Gamma)} \leq \|\psi^*\|_{H^{k+1}(\Gamma)}$.
\end{proof}
\ \\[1ex]

\begin{remark} \rm We discuss  the main results  \eqref{NotesAR_22}--\eqref{NotesAR_23a}.
The bound \eqref{NotesAR_22} is optimal for all $k \geq 2$, cf.  \eqref{NotesAR_22_kgeq2}. For the case $k=1$ we obtain 
\begin{equation} \label{AAAR}
| \psi^* - \psi_h^* |_1 \lesssim h \| \psi^* \|_{H^{2}(\Gamma)} + h \| \phi^* \|_{H^{1}(\Gamma)} \lesssim h \| \psi^* \|_{H^{3}(\Gamma)}.
\end{equation}
We note that for $k=1$  the analysis in  \cite{Falk1980} does \emph{not} yield an optimal order estimate of the form $|\psi^\ast - \psi_h^\ast|_1 \leq ch$. The result \eqref{AAAR} is optimal w.r.t. the order of convergence in $h$, but suboptimal in the sense that it involves the smoothness term $\| \psi^* \|_{H^{3}(\Gamma)}$ instead of the optimal $\| \psi^* \|_{H^{2}(\Gamma)}$. This is caused by the fact that we need a minimal regularity $\bu^\ast \in H^2(\Gamma)^3$, i.e., $\psi^\ast \in H^3(\Gamma)$.  The improvement of our result (for $k=1$) compared to \cite{Falk1980} is probably due to the fact that in our case all bilinear forms are symmetric, which is not assumed in the  framework presented in \cite{Falk1980}.

The second error bound \eqref{NotesAR_23b} is optimal if we have sufficient smoothness:  for the case $m \geq k+3$, we obtain the optimal error bound
\begin{align}
\label{NotesAR_23b_high_regularity}
\| \phi^* - \phi_h^* \|_A &\lesssim h^{k} \big( \| \psi^* \|_{H^{k+1}(\Gamma)} + \| \phi^* \|_{H^{k+1}(\Gamma)} \big).
\end{align} 
We note that the results in  \cite{Falk1980} do \emph{not} yield  optimal error bounds for $ \phi^* - \phi_h^*$ (even not under strong smoothness assumptions). 
Assuming minimal regularity, i.e.,   $m=3$, we obtain
\[
 \| \phi^* - \phi_h^* \|_A \lesssim h^k \| \psi^* \|_{H^{k+1}(\Gamma)} + \| \phi^* \|_{H^{1}(\Gamma)}  \quad\text{for}~k=1,2,
\]
which is not optimal.

The third error bound \eqref{NotesAR_23a} is suboptimal. For the case with strong smoothness assumptions ($m \geq k+3$) we  obtain the error bound
\begin{align*}
\| \phi^* - \phi_h^* \|_{L^2(\Gamma)} &\lesssim h^{k} \| \psi^* \|_{H^{k+1}(\Gamma)} + h^{k+1} \| \phi^* \|_{H^{k+1}(\Gamma)}, 
\end{align*}
which is one order lower than the optimal $h^{k+1}$ bound. 
\end{remark}
\ \\[1ex]

\section{Error analysis for the reconstruction of velocity and pressure}\label{SectPressureVelocity}
In this section we derive discretization error bounds for the discrete reconstructions $\bu_h$, $p_h$ of the velocity $\bu^\ast$ and $p^\ast$, \emph{given the discrete stream function  $\psi_h^*$}, which is the solution of \eqref{NotesAR_10}. With $\psi^*$ we denote the unique stream function of $\bu^*$. Again we make the simplifying assumption $\Gamma_h=\Gamma$.
\subsection{Error analysis for velocity reconstruction} 
The velocity reconstruction is based on the variational problem \eqref{Variationsformulierung_u}. For the discretization we consider the case with $\Gamma_h=\Gamma$, i.e. (cf. \eqref{Variationsformulierung_u_diskret}): 
Determine $\bu_h \in \left( V_{h,k_u} \right)^3$ such that
\begin{equation}
\label{Variationsformulierung_u_diskretA} 
\begin{split}
  M_h(\bu_h , \bv_h) & : = \int_{\Gamma} \bu_h \cdot \bv_h \, ds + s_h(\bu_h,\bv_h) = \int_{\Gamma}  \bg_h \cdot \bv_h \,ds \quad \forall~\bv_h \in \left( V_{h,k_u} \right)^3, \\
 \text{with}~ & ~s_h(\bu_h,\bv_h)= \rho_u \int_{ \Omega^\Gamma_h} \left( \nabla \bu \bn\right) \cdot \left( \nabla \bu \bn\right)  dx, ~ ~ \bg_h:= (\tilde{\bn}_h \times \nablaG \psi_h^*),
\end{split}
\end{equation} 
and $\psi_h^\ast \in V_{h,k}$ the solution of \eqref{NotesAR_10}.
We denote the unique solution of \eqref{Variationsformulierung_u_diskretA} by $\bu_h^*$. The exact velocity solution $\bu^\ast$ (extended constantly along normals) satisfies, cf.  \eqref{Variationsformulierung_u},
\[ \begin{split}
  M_h(\bu^\ast, \bv_h) & =\int_{\Gamma}  \bg \cdot \bv_h \,ds \quad \forall~\bv_h \in \left( V_{h,k_u} \right)^3, \\
\text{with}~~ \bg & = ( {\bn} \times \nablaG \psi^*).
\end{split} \]
Due to \eqref{lem_NotesAR_3_help}, $\|\cdot\|_M:=M_h(\cdot,\cdot)^\frac12$ defines a norm on $\left( V_{h,k_u} \right)^3$.  Using a standard Strang argument we obtain the following result.
\begin{lemma}
Let $\bu_h^*$ be the unique solution of \eqref{Variationsformulierung_u_diskretA}. The following error estimate holds:
\begin{equation} \label{errStrang}
\|\bu^\ast - \bu_h^*\|_M \leq 2 \min_{\bv_h \in \left( V_{h,k_u} \right)^3} \|\bu^\ast - \bv_h\|_M +  \|\bg - \bg_h\|_{L^2(\Gamma)}.
\end{equation}
\end{lemma}
\ \\[1ex]

The approximation error part in \eqref{errStrang} can be bounded by standard interpolation error bounds, available for trace finite elements.
\begin{lemma}
\label{lem_apprbound}
 Take $\bu \in H^{k_u+1}(\Gamma)$. The following holds:
\begin{equation} \label{apprbound}
\min_{\bv_h \in \left( V_{h,k_u} \right)^3} \|\bu - \bv_h\|_M \lesssim h^{k_u+1} \|\bu\|_{H^{k_u+1}(\Gamma)}.
\end{equation}
\end{lemma}
\begin{proof} Let $\bu \in H^{k_u+1}(\Gamma)$ be given. Its constant extension along normals $\bn$ is also denoted by $\bu$. The standard (componentwise) nodal interpolation in  $V_{h,k_u}$ is denoted by $I^{k_u}$. Note that
\begin{align}
\nonumber
& \min_{\bv_h \in \left( V_{h,k_u} \right)^3} \|\bu - \bv_h\|_M \leq \|\bu - I^{k_u} \bu \|_M \\
\label{u_energie_approx}
&\leq \| \bu - I^{k_u} \bu \|_{L^2(\Gamma)} + \left( s_h \left( \bu - I^{k_u} \bu , \bu - I^{k_u} \bu \right) \right)^{\frac{1}{2}}.
\end{align}
The first term of \eqref{u_energie_approx} can be estimated with a standard interpolation error result for trace finite elements \cite{reusken2015analysis}: 
$ \| \bu - I^{k_u} \bu \|_{L^2(\Gamma)} \lesssim h^{k_u+1} \| \bu \|_{H^{k_u+1}(\Gamma)}$. For the second term in \eqref{u_energie_approx} we get
\begin{equation} \label{est6} \begin{split}
\left( s_h \left( \bu - I^{k_u} \bu , \bu - I^{k_u} \bu \right) \right)^{\frac{1}{2}} &\lesssim h^{\frac{1}{2}} \| \nabla \left( \bu - I^{k_u} \bu \right) \|_{L^2(\Omega_h^\Gamma)} \\
&\lesssim  h^{k_u+\frac12} \| \bu \|_{H^{k_u+1}(\Omega_h^\Gamma)} 
\lesssim  h^{k_u+1} \| \bu \|_{H^{k_u+1}(\Gamma)}.
\end{split}
\end{equation}
Combing these results we obtain the bound \eqref{apprbound}.
\end{proof}
\ \\[1ex]

\begin{remark}
\label{rem_approxU} \rm
In the proof above, cf. \eqref{est6}, a scaling $\rho_u \lesssim h$ is essential to obtain the approximation error bound as  in \eqref{apprbound}.
 For  other usual choices $\rho_u \sim 1$ and $\rho_u \sim h^{-1}$ we obtain an approximation error bound $\lesssim h^{k_u+\frac{1}{2}}$ and $\lesssim h^{k_u}$, respectively. Note that the norm $\|\cdot\|_M$ depends on $\rho_u$. The scaling $\rho_u \sim h$ is optimal it the sense that it balances the two terms in \eqref{u_energie_approx}.  Without stabilization, i.e., $\rho_u=0$,  we  obtain an optimal approximation error bound in the $\|\cdot\|_{L^2(\Gamma_h)}$-norm. Below we see that an optimal error bound in the $H^1$-norm can only be derived for the case that the error in the $\|\cdot\|_M$-norm is bounded by $h^{k_u+1}$ \emph{and} $\rho_u \gtrsim h$ is satisfied, cf. Remark~\ref{rem_uH1}. Based on these observations we take the scaling $\rho_u \sim h$ in \eqref{choicerho}. In Section \ref{sectvarrho} we show results of experiments with different scalings of the stabilization parameter $\rho_u$, which confirm these findings.
\end{remark}
\ \\[1ex]

From the analysis below and from numerical experiments we conclude that for the case $k=k_u$ we have suboptimal discretization errors bounds for the velocity approximaiton $\bu_h \in (V_{h,k_u})^3$. Hence, the case $k<k_u$ is not of interest. In view of this and to simplify the presentation, \emph{in  remainder we restrict to} $k\geq k_u$.

\begin{theorem}
\label{thmdiskretisierungsfehler_u_energy}
Let $\bu^\ast$ and $\bu_h^* \in \left( V_{h,k_u} \right)^3$ be the unique solutions of \eqref{Variationsformulierung_u} and \eqref{Variationsformulierung_u_diskretA}, respectively.
We assume that $\bu^\ast \in H^{k_u+1}(\Gamma)^3$ holds. Let $k_g \in \Bbb{N}$ be such that 
 \[ \|\bn - \tilde \bn_h\|_{L^\infty(\Gamma)} \lesssim h^{k_g}.
 \]
 With $k^*:=0$ if $k=1$ and $k^*:=1$ if $k \geq 2$, the following discretization error bound holds:
\begin{equation} \label{AARR}
\|\bu^\ast - \bu_h^\ast\|_M \lesssim \, \left( h^{k_g} + h^{\min\{ k_u+1, k \} } + h^{k_u+k^*} \right)  \|\bu^\ast\|_{H^{k_u+1}(\Gamma)}.
\end{equation}
\end{theorem}
\begin{proof}
Using the Cauchy-Schwarz inequality, the assumption on the normal approximation and the discretization error for the stream function \eqref{NotesAR_22} in Theorem \ref{thmDiscretizationErrorStreamfunction} (with $m=k_u+1$) yields
\begin{align}
\notag
\|\bg - \bg_h\|_{L^2(\Gamma)} \leq & \, \| \left( \bn - \tilde{\bn}_h \right) \times \nablaG \psi^\ast \|_{L^2(\Gamma)} + \| \tilde \bn_h \times \left( \nablaG \left( \psi^\ast - \psi_h^\ast \right) \right) \|_{L^2(\Gamma)} \\
\notag
\lesssim & \, h^{k_g} \| \psi^\ast \|_{H^{1}(\Gamma)} + | \psi^* - \psi_h^* |_1 \\
\notag
\lesssim & \, h^{k_g} \| \psi^\ast \|_{H^{1}(\Gamma)} + h^{\min\{ k_u+1, k \} } \| \psi^* \|_{H^{\min \{ k_u+2, k+1 \}}(\Gamma)} \\
\notag
&+ h^{k_u+k^*} \| \psi^* \|_{H^{k_u+2}(\Gamma)} \\
\label{konsistenz_u}
\lesssim & \, \left( h^{k_g} + h^{\min\{ k_u+1, k \} } + h^{k_u+k^*} \right)  \|\bu^\ast\|_{H^{k_u+1}(\Gamma)}.
\end{align}
Here we used the assumptions $k \geq k_u$, $\psi^\ast \in H^1_*(\Gamma) \cap H^{k_u+2}(\Gamma)$, and $\|\phi^*\|_{k_u} \lesssim \|\psi^*\|_{k_u+2} \lesssim \|\bu^*\|_{k_u+1}$, due to the first estimate in \eqref{NormConnection_U_Psi}.
The result \eqref{AARR} follows from the approximation error \eqref{apprbound}, the error bound \eqref{konsistenz_u}  and the Strang estimate \eqref{errStrang}.
\end{proof}
\ \\[1ex]
We discuss this theorem in the next remark.
\begin{remark} \rm
\label{rem_connection_up}
The discretization error bound in Theorem \ref{thmdiskretisierungsfehler_u_energy} implies that the choices $k=k_u+1$ and $k_g=k_u+1$ lead to optimal order of convergence:
\begin{align*}
\|\bu^\ast - \bu_h^\ast\|_M \lesssim \, h^{k_u+1} \|\bu^\ast\|_{H^{k_u+1}(\Gamma)}.
\end{align*}
In particular,  for linear finite elements for the velocity ($k_u=1$) we obtain an optimal discretization error bound $\lesssim h^2$ if we use quadratic finite elements for the stream function ($k=2$) and a  normal approximation with second order accuracy ($k_g=2$). Choosing either $k=1$ or $k_g=1$ will lead to an error bound $\lesssim h$ and is therefore suboptimal. This suboptimality is confirmed in  numerical experiments.
\end{remark}
\ \\[1ex]

As an immediate consequence we obtain the following optimal discretization error bound in the norm $\|\cdot \|_{L^2(\Gamma)}$.
\begin{corollary}
\label{thmdiskretisierungsfehler_u_L2}
If the assumptions as in Theorem~\ref{thmdiskretisierungsfehler_u_energy} are satisfied and we take $k=k_g=k_u+1$, the following error bound holds:
\begin{align*}
\|\bu^\ast - \bu_h^\ast\|_{L^2(\Gamma)} \lesssim h^{k_u+1}   \|\bu^\ast\|_{H^{k_u+1}(\Gamma)}.
\end{align*}
\end{corollary}
\ \\[1ex]
In the next theorem we derive a discretization error bound in the $\|\cdot\|_{H^1(\Gamma)}$-norm.
\begin{theorem}
\label{thmdiskretisierungsfehler_u_H1}
Let the assumptions as in Theorem~\ref{thmdiskretisierungsfehler_u_energy} be satisfied. The following error bound holds:
\begin{align*}
\|\bu^\ast - \bu_h^\ast\|_{H^1(\Gamma)} \lesssim \left( h^{k_g-1} + h^{\min\{ k_u, k-1 \} } + h^{k_u+k^*-1} \right)  \|\bu^\ast\|_{H^{k_u+1}(\Gamma)}.
\end{align*}
\end{theorem}
\begin{proof}
We consider the splitting (with a constant extension of $\bu^\ast$)
\begin{align}
\label{u_h1fehler_split}
\| \nabla_\Gamma \left( \bu^* - \bu_h^* \right) \|_{L^2(\Gamma)} \leq \| \nabla \left( \bu^* - I^{k_u} \bu^* \right) \|_{L^2(\Gamma)} + \| \nabla \left( I^{k_u} \bu^* - \bu_h^* \right) \|_{L^2(\Gamma)},
\end{align}
with $I^{k_u}$ the nodal interpolation operator as used in the proof of Lemma~\ref{lem_apprbound}.
We use the interpolation error bound
\begin{equation} \label{int2} \| \nabla \left( \bu^* - I^{k_u} \bu^* \right) \|_{L^2(\Gamma)} \lesssim h^{k_u} \|\bu\|_{H^{k_u+1}(\Gamma)}.
\end{equation}
For the other term we first recall the following result, for $T\in \cT_h^\Gamma$:
\[
 \|v\|_{L^2(\Gamma\cap T)}^2 \lesssim \left( h_T^{-1}\|v\|_{L^2(T)}^2 + h_T \|\nabla v\|_{L^2(T)}^2 \right), \quad \text{for all}~v \in H^1(T),
\]
with $h_T:={\rm diam}(T)$, cf.~\cite{Hansbo02}.
Using this, a standard inverse inequality and the estimate \eqref{lem_NotesAR_3_help} yields
\begin{align*}
\| \nabla \left( I^{k_u} \bu^* - \bu_h^* \right) \|_{L^2(\Gamma)} & \lesssim h^{-\frac12}\| \nabla \left( I^{k_u} \bu^* - \bu_h^* \right) \|_{L^2(\Omega^\Gamma_h)} 
 \lesssim h^{-1\frac12} \| I^{k_u} \bu^* - \bu_h^* \|_{L^2(\Omega^\Gamma_h)} \\
& \lesssim h^{-1} \| I^{k_u} \bu^* - \bu_h^* \|_{L^2(\Gamma)} +h^{-\frac12} \| \bn \cdot \nabla \left( I^{k_u} \bu^* - \bu_h^* \right) \|_{L^2(\Omega^\Gamma_h)} \\
& \lesssim h^{-1} \| I^{k_u} \bu^* - \bu_h^*  \|_M.
\end{align*}
We apply the triangle equality, use the interpolation error bound derived in the proof of Lemma~\ref{lem_apprbound}, and the discretization error from Theorem \ref{thmdiskretisierungsfehler_u_energy}. This yields
\begin{align*}
& \| \nabla \left( I^{k_u} \bu^* - \bu_h^* \right) \|_{L^2(\Gamma)} \lesssim  \, h^{-1} \| \bu^* - I^{k_u} \bu^* \|_M + h^{-1} \| \bu^* - \bu_h^*  \|_M \\
 & \lesssim  \, \left(h^{k_u}+ h^{k_g-1} + h^{\min\{ k_u, k-1 \} } + h^{k_u+k^*-1} \right)  \|\bu^\ast\|_{H^{k_u+1}(\Gamma)}.
\end{align*}
Combining this with \eqref{int2}, and noting that $k_u \geq \min\{ k_u, k-1 \}$, completes  the proof.
\end{proof}
\begin{remark} \rm
\label{rem_uH1}
In the proof above the result \eqref{lem_NotesAR_3_help} is needed. For this it is essential that in the stabilization in the discrete problem \eqref{Variationsformulierung_u_diskretA} we use a parameter $\rho_u \gtrsim h$. Hence, without stabilization ($\rho_u=0$) we do not obtain an optimal order discretization error bound in the $\|\cdot\|_{H^1(\Gamma_h)}$-norm. Numerical  experiments in Section~\ref{SectExp} show that for optimal order convergence the $\|\cdot\|_{H^1(\Gamma_h)}$-norm  the normal derivative volume stabilization is  indeed essential. 
\end{remark}

\subsection{Error analysis for pressure reconstruction} 
The pressure reconstruction is based on the variational problem \eqref{Variationsformulierung_p}. For the discretization we consider the case with $\Gamma_h=\Gamma$, i.e. (cf. \eqref{Variationsformulierung_p_diskret}): 
Determine $p_h \in V_{h,k_p}$ with $\int_{\Gamma_h} p_h \, ds_h=0$,  such that
\begin{equation}
\label{Variationsformulierung_p_diskretA} 
\begin{split}
A_h(p_h , \xi_h) & : = \int_{\Gamma} \nablaG p_h \cdot \nablaG \xi_h \, ds + s_h(p_h,\xi_h) = \int_{\Gamma}  \bz_h \cdot \nablaG \xi_h \,ds \quad \forall~\xi_h \in V_{h,k_p}, \\
 \text{with}~~s_h(p_h,\xi_h)&= \rho_p \int_{ \Omega^\Gamma_h} \left( \bn \cdot \nabla p_h \right) \cdot\left( \bn \cdot \nabla \xi_h \right) \, dx, \quad  \bz_h:=  (K \RotG \psi_h^* + \blf ),
\end{split}
\end{equation} 
and $\rho_p$ as in \eqref{stabrho}.
We denote the unique solution of \eqref{Variationsformulierung_p_diskretA} with $p_h^*$. The exact pressure solution $p^\ast$ (extended constantly along normals) satisfies, cf.  \eqref{Variationsformulierung_p},
\[ \begin{split}
  A_h(p^\ast, \xi_h) & =\int_{\Gamma} \bz \cdot \nablaG \xi_h \,ds \quad \forall~\xi_h \in V_{h,k_p}, \\
\text{with}~~ \bz & =  (K \RotG \psi^* + \blf ).
\end{split} \]
The corresponding seminorm is denoted by $\|\cdot\|_A$ (which is the same as in Section~\ref{SectAnalysisStreamfunction} but with $\rho_p$ instead of $\rho$). Using a standard Strang argument we obtain the following result.
\begin{lemma}
Let $p_h^*$ be the unique solution of \eqref{Variationsformulierung_p_diskretA}. The following error estimate holds:
\begin{equation} \label{errStrang_p}
\|p^\ast - p_h^*\|_A \leq 2 \min_{\xi_h \in V_{h,k_p}} \|p^\ast - \xi_h\|_A +  \|\bz - \bz_h\|_{L^2(\Gamma)}.
\end{equation}
\end{lemma}
\ \\[1ex]
For the approximation error part in \eqref{errStrang_p} we have a bound as  in Proposition \ref{propestimates}. For $p^\ast \in H^1_*(\Gamma) \cap H^{k_p+1}(\Gamma)$ the following holds:
\begin{align}
\label{apprbound_p}
\min_{\xi_h \in V_{h,k_p}} \|p^\ast - \xi_h\|_A \lesssim h^{k_p} \|p^\ast\|_{H^{k_p+1}(\Gamma)}.
\end{align}

We now present a discretization error bound for the pressure.
\begin{theorem}
\label{thmdiskretisierungsfehler_p_energy}
Let $p^\ast \in H^1_*(\Gamma)$ and $p_h^* \in V_{h,k_p}$ be the unique solutions of \eqref{Variationsformulierung_p} and \eqref{Variationsformulierung_p_diskretA}, respectively. We assume $p^\ast \in  H^{k_p+1}(\Gamma)$, and  let $m \geq 3$ be such that $\psi^* \in H^m(\Gamma)$.  With $s_1=\min \{m, k+1\}$, $s_2=\min \{m-2, k+1\}$ and $k^*:=0$ if $k=1$ and $k^*:=1$ if $k \geq 2$ the following error bound holds:
\begin{align*}
\|p^\ast - p_h^\ast\|_A \lesssim \, h^{k_p} \|p^\ast\|_{H^{k_p+1}(\Gamma)} + h^{s_1-1} \| \psi^* \|_{H^{s_1}(\Gamma)} + h^{s_2+k^*} \| \psi^* \|_{H^{s_2+2}(\Gamma)}.
\end{align*}
\end{theorem}
\begin{proof}
Using  the discretization error bound for the stream function \eqref{NotesAR_22} in Theorem \ref{thmDiscretizationErrorStreamfunction} and $\|\phi^*\|_{r} \lesssim \|\psi^*\|_{r+2}$ we obtain
\begin{align*}
\|\bz - \bz_h\|_{L^2(\Gamma)} \leq & \, \| K \|_{L^\infty(\Gamma)} \| \RotG \left( \psi^* - \psi_h^* \right) \|_{L^2(\Gamma)} \\
\lesssim & \, | \psi^* - \psi_h^* |_1  \lesssim  \, h^{s_1-1} \| \psi^* \|_{H^{s_1}(\Gamma)} + h^{s_2+k^*} \| \phi^* \|_{H^{s_2}(\Gamma)} \\
\lesssim & \, h^{s_1-1} \| \psi^* \|_{H^{s_1}(\Gamma)} + h^{s_2+k^*} \| \psi^* \|_{H^{s_2+2}(\Gamma)}.
\end{align*}
 Combining this with  the approximation error \eqref{apprbound_p} and the Strang  estimate \eqref{errStrang_p} completes the proof. 
\end{proof}

If $p^\ast \in H^{k_p+1}(\Gamma)$ holds, it is reasonable to assume the smoothness property $\bu^\ast \in H^{k_p+2}(\Gamma)$, i.e., $\psi^\ast \in H^{k_p+3}(\Gamma)$. We consider the special case of Theorem~\ref{thmdiskretisierungsfehler_p_energy} with $m=k_p+3$:
\begin{corollary} Let the assumptions as in Theorem~\ref{thmdiskretisierungsfehler_p_energy} be satisfied, with $m=k_p+3$. We take $k \in \{k_p,k_p+1\}$. Then the following error bound holds:
\[
 \|p^\ast - p_h^\ast\|_A \lesssim \, h^{k_p}\left( \|p^\ast\|_{H^{k_p+1}(\Gamma)} +\|\bu^\ast\|_{H^{k_p+2}(\Gamma)}\right).
\]

\end{corollary}

\begin{remark} \label{commentassumption}\rm In the error analysis of this paper we used the assumption $\Gamma_h=\Gamma$.  Using the techniques available in the literature for scalar elliptic surface PDEs, an error analysis including the geometry approximation can be developed. Such an analysis, however, will be very technical. Based on the analysis in  \cite{GrandeLehrenfeldReusken} of the higher order parametric finite element method explained in Remark~\ref{RemGammah}, we expect optimal order results for the velocity approximation if for the geometry approximation we use the same polynomial degree $k_u$ (i.e., \emph{iso}parametric w.r.t the velocity finite element space).  Using this isoparametric approach and $k=k_u+1$, $k_g=k_u+1$, numerical experiments with the parametric finite element method of Remark~\ref{RemGammah} show that we indeed obtain optimal results $\|\psi^*-\psi^*_h\|_A \sim h^k$, $\|\bu - \bu_h\|_{H^1(\Gamma_h)} \sim h^{k_u}$, $\|\bu - \bu_h\|_{L^2(\Gamma_h)} \sim h^{k_u+1}$, but  a 
suboptimal error $\|\psi^*-\psi^*_h\|_{L^2(\Gamma_h)} \sim h^k$ (due to the geometrical error).
\end{remark}

\section{Numerical experiments} \label{SectExp}
For the implementation of the method we used Netgen/NGSolve with ngsxfem \cite{Netgen, ngsxfem}. We consider the unit sphere $ \Gamma \subset \Omega:=[-2,2]^3$:
\begin{align*}
\Gamma:=\{\mathbf{x} \in \Bbb{R}^3: \Phi(\bx):= \sqrt{x_1^2 + x_2^2 + x_3^2-1}=0 \}.
\end{align*}
 We choose the smooth solutions 
\begin{align*}
\bu&:=   \left( \begin {array}{c} {\frac {x_{{1}} \left( 6x_{{2}}-x_{{3}}
 \right) \sqrt {{x_{{1}}}^{2}+{x_{{2}}}^{2}+{x_{{3}}}^{2}}-x_{{2}}\cos
 \left( 6 \right)  \left( {x_{{2}}}^{2}-2\,{x_{{3}}}^{2} \right) }{
 \left( {x_{{1}}}^{2}+{x_{{2}}}^{2}+{x_{{3}}}^{2} \right) ^{3/2}}}
\\ \noalign{\medskip}-{\frac {-\cos \left( 6 \right) x_{{1}}{x_{{2}}}^
{2}+ \left( 6{x_{{1}}}^{2}-6{x_{{3}}}^{2}-x_{{2}}x_{{3}} \right) 
\sqrt {{x_{{1}}}^{2}+{x_{{2}}}^{2}+{x_{{3}}}^{2}}}{ \left( {x_{{1}}}^{
2}+{x_{{2}}}^{2}+{x_{{3}}}^{2} \right) ^{3/2}}}\\ \noalign{\medskip}-{
\frac { \left( 6x_{{2}}x_{{3}}-{x_{{1}}}^{2}+{x_{{2}}}^{2} \right) 
\sqrt {{x_{{1}}}^{2}+{x_{{2}}}^{2}+{x_{{3}}}^{2}}+2\,\cos \left( 6
 \right) x_{{1}}x_{{2}}x_{{3}}}{ \left( {x_{{1}}}^{2}+{x_{{2}}}^{2}+{x
_{{3}}}^{2} \right) ^{3/2}}}\end {array} \right) ,\\
p&:={\frac { \left( {6}^{3}x_{{2}}x_{{3}}+x_{{1}}\cos \left( {6}^{2}
 \right) \sqrt {{x_{{1}}}^{2}+{x_{{2}}}^{2}+{x_{{3}}}^{2}} \right) x_{
{1}}}{ \left( {x_{{1}}}^{2}+{x_{{2}}}^{2}+{x_{{3}}}^{2} \right) ^{3/2}
}},
\end{align*}
of the Stokes equation \eqref{strongform-1}-\eqref{strongform-2} with $\alpha=1$, and corresponding stream function and vorticity
\begin{align*}
\psi&:={\frac {x_{{1}} \left( 6x_{{3}}+x_{{2}} \right) \sqrt {{x_{{1}}}^{2}+{
x_{{2}}}^{2}+{x_{{3}}}^{2}}-{x_{{2}}}^{2}x_{{3}}\cos \left( 6 \right) 
}{ \left( {x_{{1}}}^{2}+{x_{{2}}}^{2}+{x_{{3}}}^{2} \right) ^{3/2}}}, \\
\phi&:=-6\,{\frac {x_{{1}} \left( 6x_{{3}}+x_{{2}} \right) \sqrt {{x_{{1}}}^{
2}+{x_{{2}}}^{2}+{x_{{3}}}^{2}}+1/3\,\cos \left( 6 \right) x_{{3}}
 \left( {x_{{1}}}^{2}-5\,{x_{{2}}}^{2}+{x_{{3}}}^{2} \right) }{
 \left( {x_{{1}}}^{2}+{x_{{2}}}^{2}+{x_{{3}}}^{2} \right) ^{3/2}}}.
\end{align*}
The velocity solution is tangential and divergence-free and  all solutions are extended constantly along normals. Using MAPLE we determine the corresponding right-hand side $\blf$ which is also constant in normal direction and defines  the data approximation $\mathbf{f}_h$. The velocity $\bu$ and the stream function $\psi$ are visualized in Figure \ref{Grafik_u_psi}.

\begin{figure}
	\centering
    \subfigure{\includegraphics[width=0.49\textwidth]{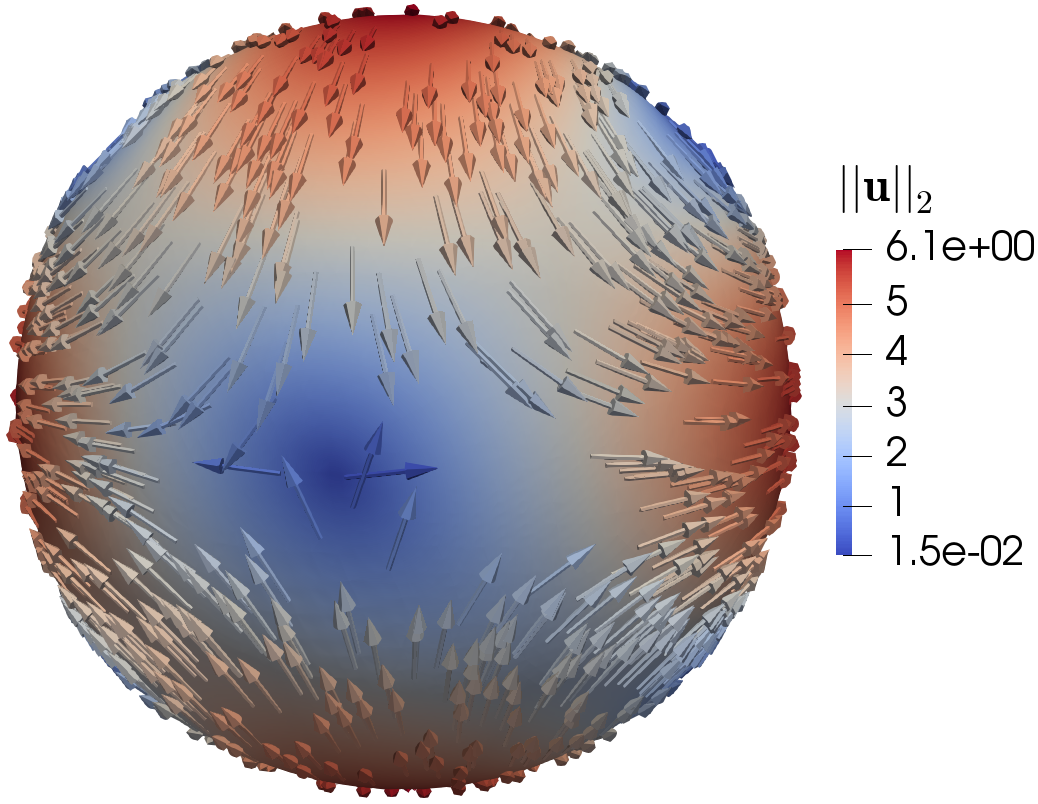}}
    \subfigure{\includegraphics[width=0.49\textwidth]{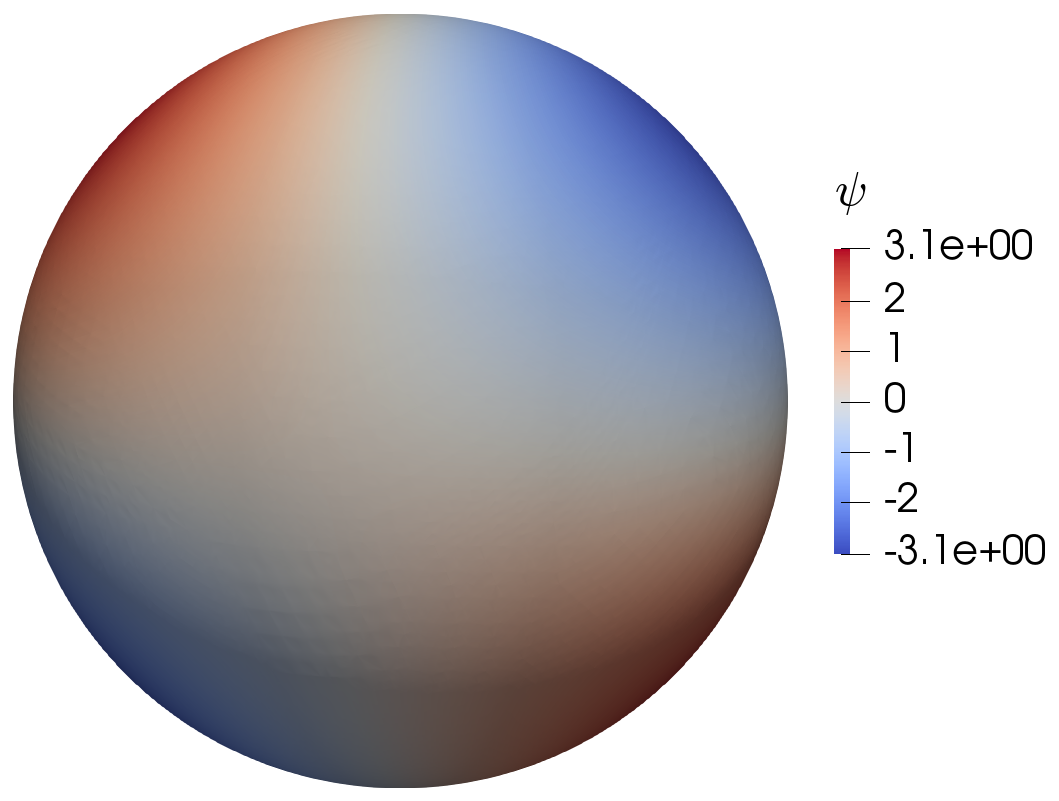}}
	\caption{The velocity $\bu$ (left) and the stream function $\psi$ (right).}
	\label{Grafik_u_psi}
\end{figure}

For the discretization we use an unstructured tetrahedral triangulation of $\Omega$ in Netgen with starting mesh size $h=0.6$. In every refinement step the mesh is locally refined using a marked-edge bisection method for the tetrahedra that are intersected by the surface \cite{Schoeberl}.
The piecewise planar surface approximation $\Gamma_h$ and the trace finite element spaces are constructed as explained in Remark~\ref{RemGammah}. The  normal on $\Gamma_h$ is defined by $\bn_h:=\frac{I^1_h \Phi}{\|I^1_h \Phi\|_2}$ and satisfies $\|\bn-\bn_h\|_\infty \lesssim h$. In the numerical experiments  a higher accuracy normal $\tilde \bn_h$ is used for the reconstruction of the velocity, unless stated otherwise. This normal is defined by $\tilde \bn_h:=\frac{I^2_h \Phi}{\|I^2_h \Phi\|_2}$ and satisfies $\|\bn-\tilde \bn_h \|_\infty \lesssim h^2$. We do not need a curvature approximation and set $K_h=K=1$. Based on the results of our analysis, cf. the discussion in  Remark \ref{rem_connection_up}, we define as the \emph{standard parameter choice}: $k=2$, $k_u=1$ and $k_p=1$. The parameters of the volume normal derivative stabilizations are set to $\rho=\rho_u=\rho_p=h$ and we use $k_g=2$, unless stated otherwise. Note that \emph{due to the geometry approximation ${\rm dist}(\Gamma_h,\Gamma) \lesssim h^2$ we can not 
expect any error to be better than $\lesssim h^2$}.

\subsection{Results of the method with standard parameter choice}
We present results for the standard method described above. In addition the case $k_p=2$ (with solution  denoted by $\tilde{p}_h$) is considered. All results are illustrated in the Figures \ref{Errors_psi_phi} -- \ref{Errors_u_p}. We observe an optimal second order convergence for the error $\|\psi_h-\psi\|_{A}$, consistent with the estimate \eqref{NotesAR_22_kgeq2}. The error $\|\psi_h-\psi\|_{L^2(\Gamma_h)}$ is not treated in the analysis but converges with second order. This suboptimality  can be explained by the geometric error $\sim h^2$. We  observe the same convergence orders for the vorticity error $\|\phi-\phi_h\|$ as for the stream function error $\|\psi-\psi_h\|$, both for $\|\cdot\|=\|\cdot\|_A$ and $\|\cdot\|=\|\cdot\|_{L^2(\Gamma_h)}$. For the velocity and the pressure \emph{optimal orders of convergence in all norms}, as predicted by the analysis (modulo geometric errors) are observed. For the error $\|\tilde p_h-p\|_{L^2(\Gamma_h)}$ we obtain suboptimal second order convergence (not shown), due to the geometry error $\sim h^2$.

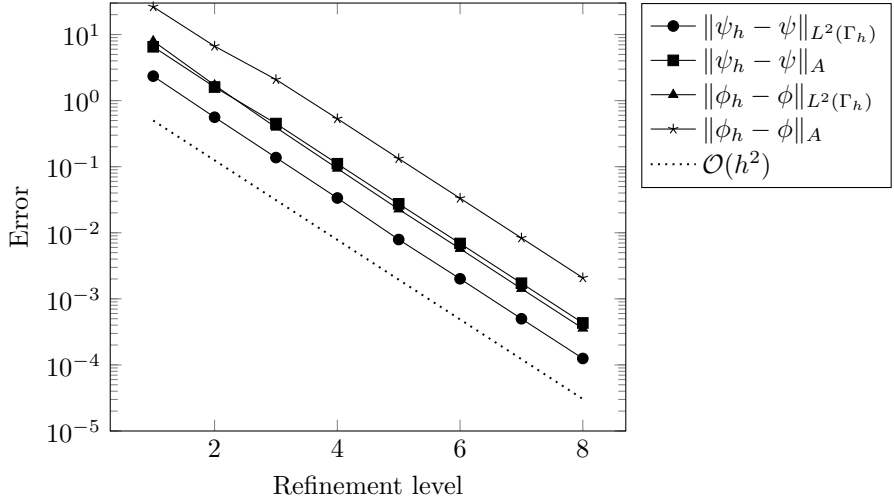
\begin{figure}
  \begin{tikzpicture}
  \def\vara{10.0}
  \def\varb{0.5}
  \begin{semilogyaxis}[ xlabel={Refinement level}, ylabel={Error}, ymin=1E-5, ymax=30, legend style={ cells={anchor=west}, legend pos=outer north east}, cycle list name=mark list ]
    \addplot table[x=level, y=1_errorL2] {Stat_Stokes_Modified_PsiPhi.dat};
    \addplot table[x=level, y=1_errorH1] {Stat_Stokes_Modified_PsiPhi.dat};
    \addplot table[x=level, y=2_errorL2] {Stat_Stokes_Modified_PsiPhi.dat};
    \addplot table[x=level, y=2_errorH1] {Stat_Stokes_Modified_PsiPhi.dat};
     \addplot[dotted,line width=0.75pt] coordinates { 
    (1,\varb) (2,\varb*0.5*0.5) (3,\varb*0.25*0.25) (4,\varb*0.125*0.125) (5,\varb*0.0625*0.0625) (6,\varb*0.03125*0.03125) (7,\varb*0.015625*0.015625) (8,\varb*0.0078125*0.0078125)
    };

    \legend{$\|\psi_h-\psi\|_{L^2(\Gamma_h)}$ , $\|\psi_h-\psi\|_{A}$ , $\|\phi_h-\phi\|_{L^2(\Gamma_h)}$ , $\|\phi_h-\phi\|_{A}$ , $\mathcal{O}(h^2)$}
  \end{semilogyaxis}
  \end{tikzpicture}
  \caption{Errors for the stream function and the vorticity; $k=2$.}
  \label{Errors_psi_phi}
\end{figure} 

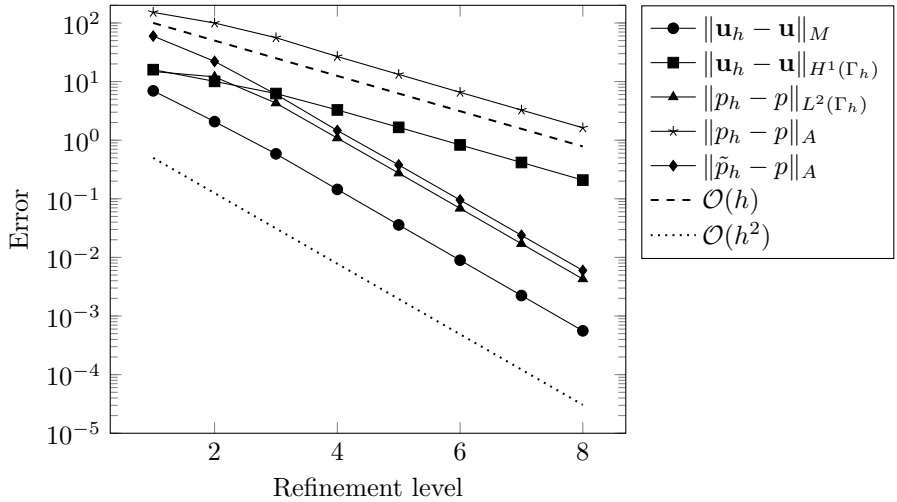
\begin{figure}
  \begin{tikzpicture}
  \def\vara{100.0}
  \def\varb{0.5}
  \begin{semilogyaxis}[ xlabel={Refinement level}, ylabel={Error}, ymin=1E-5, ymax=200, legend style={ cells={anchor=west}, legend pos=outer north east}, cycle list name=mark list ]
    \addplot table[x=level, y=1_errorL2] {Stat_Stokes_Modified_UP.dat};
    \addplot table[x=level, y=1_errorH1] {Stat_Stokes_Modified_UP.dat};
    \addplot table[x=level, y=2_errorL2] {Stat_Stokes_Modified_UP.dat};
    \addplot table[x=level, y=2_errorH1] {Stat_Stokes_Modified_UP.dat};
    \addplot table[x=level, y=3_errorH1] {Stat_Stokes_Modified_UP.dat};
    \addplot[dashed,line width=0.75pt] coordinates { 
    (1,\vara) (2,\vara*0.5) (3,\vara*0.25) (4,\vara*0.125) (5,\vara*0.0625) (6,\vara*0.03125) (7,\vara*0.015625) (8,\vara*0.0078125)
    };
     \addplot[dotted,line width=0.75pt] coordinates { 
    (1,\varb) (2,\varb*0.5*0.5) (3,\varb*0.25*0.25) (4,\varb*0.125*0.125) (5,\varb*0.0625*0.0625) (6,\varb*0.03125*0.03125) (7,\varb*0.015625*0.015625) (8,\varb*0.0078125*0.0078125)
    };

    \legend{$\|\bu_h-\bu\|_{M}$ , $\|\bu_h-\bu\|_{H^1(\Gamma_h)}$ , $\|p_h-p\|_{L^2(\Gamma_h)}$ , $\|p_h-p\|_{A}$, $\|\tilde{p}_h-p\|_{A}$ , $\mathcal{O}(h)$, $\mathcal{O}(h^2)$}
  \end{semilogyaxis}
  \end{tikzpicture}
  \caption{Errors for  velocity and pressure; $k=k_g=2$, $k_u=1$, $k_p=1$ ($p_h$) or $k_p=2$ ($\tilde p_h$).}
    \label{Errors_u_p}
\end{figure} 

\subsection{Variation of the parameter $\rho_u$} \label{sectvarrho}
In this experiment we illustrate the effect of different scalings for the parameter $\rho_u$ of the volume normal derivative stabilization. We consider $\rho_u \in \{0,h,1,h^{-1}\}$. Note that in the standard parameter setting described above we take $\rho_u=h$. Results for the velocity error in three different norms are shown  in the Figures \ref{Different_rho_M} -- \ref{Different_rho_H1}. As discussed in Remark \ref{rem_approxU}, our analysis predicts  error bounds $\|\bu_h-\bu\|_{M} \lesssim h^\frac{3}{2}$ and $\|\bu_h-\bu\|_{M} \lesssim h$ for the parameter values $\rho_u=1$ and $\rho_u=h^{-1}$, respectively. These convergence rates are observed in Figure~ \ref{Different_rho_M}. Only the choices $\rho_u\sim h$ and $\rho_u=0 $ (no stabilization) result in second order convergence $\| \bu_h-\bu \|_{M} \sim h^2$. Note that for $\rho_u=0$ the norm $\|\cdot\|_M$ coincides with the $\|\cdot\|_{L^2(\Gamma)}$ norm. The scaling $\rho_u \sim h^{-1}$ leads to suboptimal convergence of the error $\|\bu_h-\bu\|_{L^2(\Gamma_h)}$, cf. Figure~\ref{Different_rho_L2}.  The results in Figure~\ref{Different_rho_H1} show that without stabilization ($\rho_u=0$) the rate of convergence in the $\|\cdot\|_{H^1(\Gamma_h)}$-norm is suboptimal, cf. Remark \ref{rem_uH1}. Hence, in accordance with our findings based on the error analysis, these results of the numerical experiments lead to the (optimal) parameter scaling
 $\rho_u \sim h$. 

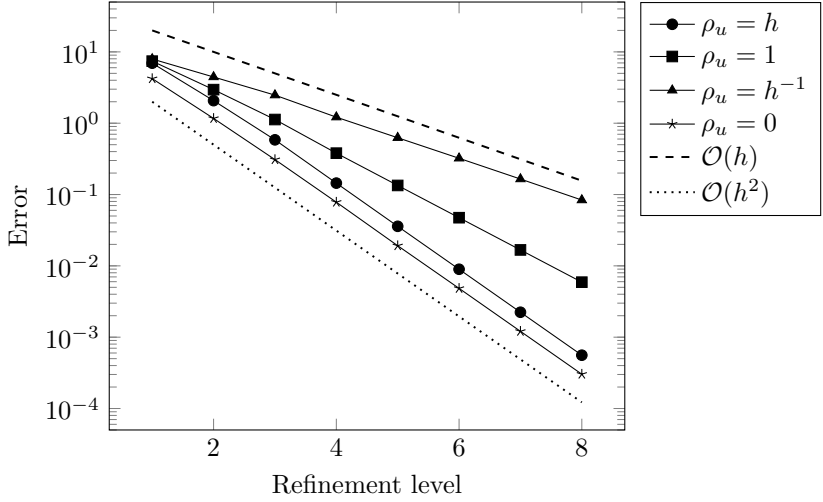
\begin{figure}
  \begin{tikzpicture}
  \def\vara{20.0}
  \def\varb{2.0}
  \begin{semilogyaxis}[ xlabel={Refinement level}, ylabel={Error}, ymin=5E-5, ymax=50, legend style={ cells={anchor=west}, legend pos=outer north east}, cycle list name=mark list ]
    \addplot table[x=level, y=rhoh_errorM] {Stat_Stokes_Modified_UrhoEnergy.dat};
    \addplot table[x=level, y=rho1_errorM] {Stat_Stokes_Modified_UrhoEnergy.dat};
    \addplot table[x=level, y=rho1dh_errorM] {Stat_Stokes_Modified_UrhoEnergy.dat};
    \addplot table[x=level, y=rho0_errorM] {Stat_Stokes_Modified_UrhoEnergy.dat};
    \addplot[dashed,line width=0.75pt] coordinates { 
    (1,\vara) (2,\vara*0.5) (3,\vara*0.25) (4,\vara*0.125) (5,\vara*0.0625) (6,\vara*0.03125) (7,\vara*0.015625) (8,\vara*0.0078125)
    };
     \addplot[dotted,line width=0.75pt] coordinates { 
    (1,\varb) (2,\varb*0.5*0.5) (3,\varb*0.25*0.25) (4,\varb*0.125*0.125) (5,\varb*0.0625*0.0625) (6,\varb*0.03125*0.03125) (7,\varb*0.015625*0.015625) (8,\varb*0.0078125*0.0078125)
    };
    \legend{$\rho_u=h$ , $\rho_u=1$ , $\rho_u=h^{-1}$ , $\rho_u=0$ , $\mathcal{O}(h)$, $\mathcal{O}(h^2)$}
  \end{semilogyaxis}
  \end{tikzpicture}
  \caption{Error $\|\bu_h-\bu\|_M$ for different  $\rho_u$ scalings; $k=k_g=2$, $k_u=1$.}
\label{Different_rho_M}
\end{figure} 

\begin{figure}
  \begin{tikzpicture}
  \def\vara{10.0}
  \def\varb{2.0}
  \begin{semilogyaxis}[ xlabel={Refinement level}, ylabel={Error}, ymin=5E-5, ymax=20, legend style={ cells={anchor=west}, legend pos=outer north east}, cycle list name=mark list ]
    \addplot table[x=level, y=rhoh_errorL2] {Stat_Stokes_Modified_UrhoL2.dat};
    \addplot table[x=level, y=rho1_errorL2] {Stat_Stokes_Modified_UrhoL2.dat};
    \addplot table[x=level, y=rho1dh_errorL2] {Stat_Stokes_Modified_UrhoL2.dat};
    \addplot table[x=level, y=rho0_errorL2] {Stat_Stokes_Modified_UrhoL2.dat};
    \addplot[dashed,line width=0.75pt] coordinates { 
    (1,\vara) (2,\vara*0.5) (3,\vara*0.25) (4,\vara*0.125) (5,\vara*0.0625) (6,\vara*0.03125) (7,\vara*0.015625) (8,\vara*0.0078125)
    };
     \addplot[dotted,line width=0.75pt] coordinates { 
    (1,\varb) (2,\varb*0.5*0.5) (3,\varb*0.25*0.25) (4,\varb*0.125*0.125) (5,\varb*0.0625*0.0625) (6,\varb*0.03125*0.03125) (7,\varb*0.015625*0.015625) (8,\varb*0.0078125*0.0078125)
    };

    \legend{$\rho_u=h$ , $\rho_u=1$ , $\rho_u=h^{-1}$ , $\rho_u=0$ , $\mathcal{O}(h)$, $\mathcal{O}(h^2)$}
  \end{semilogyaxis}
  \end{tikzpicture}
  \caption{Error $\|\bu_h-\bu\|_{L^2(\Gamma_h)}$ for different $\rho_u$ scalings; $k=k_g=2$, $k_u=1$.}
  \label{Different_rho_L2}
\end{figure} 

\begin{figure}
  \begin{tikzpicture}
  \def\vara{10.0}
  \def\varb{2.0}
  \begin{semilogyaxis}[ xlabel={Refinement level}, ylabel={Error}, ymin=5E-2, ymax=80, legend style={ cells={anchor=west}, legend pos=outer north east}, cycle list name=mark list ]
    \addplot table[x=level, y=rhoh_errorH1] {Stat_Stokes_Modified_UrhoH1.dat};
    \addplot table[x=level, y=rho1_errorH1] {Stat_Stokes_Modified_UrhoH1.dat};
    \addplot table[x=level, y=rho1dh_errorH1] {Stat_Stokes_Modified_UrhoH1.dat};
    \addplot table[x=level, y=rho0_errorH1] {Stat_Stokes_Modified_UrhoH1.dat};
    \addplot[dashed,line width=0.75pt] coordinates { 
    (1,\vara) (2,\vara*0.5) (3,\vara*0.25) (4,\vara*0.125) (5,\vara*0.0625) (6,\vara*0.03125) (7,\vara*0.015625) (8,\vara*0.0078125)
    };
    \legend{$\rho_u=h$ , $\rho_u=1$ , $\rho_u=h^{-1}$ , $\rho_u=0$ , $\mathcal{O}(h)$, $\mathcal{O}(h^2)$}
  \end{semilogyaxis}
  \end{tikzpicture}
  \caption{Error $\|\bu_h-\bu\|_{H^1(\Gamma_h)}$ for different $\rho_u$ scalings; $k=k_g=2$, $k_u=1$.}
  \label{Different_rho_H1}
\end{figure}
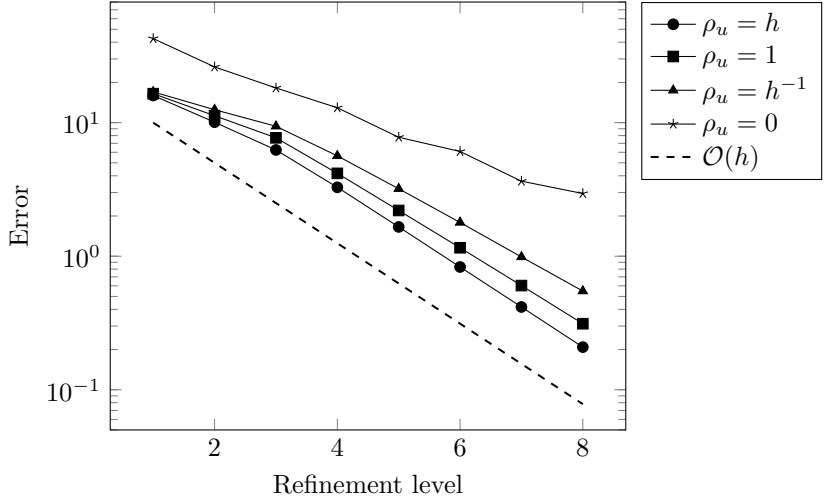 

\subsection{Choice of the normal in the reconstruction of the velocity}
In the reconstruction of the velocity \eqref{Variationsformulierung_u_diskret} we introduced an approximate normal $\tilde \bn_h$. The accuracy of this normal is described by the order parameter $k_g$, cf. Theorem~\ref{thmdiskretisierungsfehler_u_energy}. Our error analysis resulted in the parameter choice $k_g=k_u+1$, cf. Remark~\ref{rem_connection_up}. In the experiments above we used $k_u=1$, $k_g=2$. 
 We performed an experiment in which the normal $\tilde \bn_h$, that is used in the experiments above (with $k_g=2$), is replaced by the normal $\bn_h$ (with $k_g=1$). Results are shown in  Figure~\ref{Error_normals}. We observe that an optimal convergence order in both the $\| \cdot \|_M$-norm and the $\|\bu_h-\bu\|_{H^1(\Gamma_h)}$ is \emph{not}  obtained if we use $\bn_h$. We then lose one order of convergence in the $\| \cdot \|_M$-norm.

\begin{figure}
  \begin{tikzpicture}
  \def\vara{10.0}
  \def\varb{2.0}
  \begin{semilogyaxis}[ xlabel={Refinement level}, ylabel={Error}, ymin=5E-5, ymax=50, legend style={ cells={anchor=west}, legend pos=outer north east}, cycle list name=mark list ]
    \addplot table[x=level, y=1_errorL2] {Stat_Stokes_Modified_U_Normals.dat};
    \addplot table[x=level, y=1_errorH1] {Stat_Stokes_Modified_U_Normals.dat};
    \addplot table[x=level, y=2_errorL2] {Stat_Stokes_Modified_U_Normals.dat};
    \addplot table[x=level, y=2_errorH1] {Stat_Stokes_Modified_U_Normals.dat};
    \addplot[dashed,line width=0.75pt] coordinates { 
    (1,\vara) (2,\vara*0.5) (3,\vara*0.25) (4,\vara*0.125) (5,\vara*0.0625) (6,\vara*0.03125) (7,\vara*0.015625) (8,\vara*0.0078125)
    };

     \addplot[dotted,line width=0.75pt] coordinates { 
    (1,\varb) (2,\varb*0.5*0.5) (3,\varb*0.25*0.25) (4,\varb*0.125*0.125) (5,\varb*0.0625*0.0625) (6,\varb*0.03125*0.03125) (7,\varb*0.015625*0.015625) (8,\varb*0.0078125*0.0078125)
    };

    \legend{ \small $\|\bu_h-\bu\|_{M}$ with $\tilde \bn_h$ , \small $\|\bu_h-\bu\|_{H^1(\Gamma_h)}$ with $\tilde \bn_h$ , \small $\|\bu_h-\bu\|_{M}$ with $\bn_h$ , \small $\|\bu_h-\bu\|_{H^1(\Gamma_h)}$ with $\bn_h$  , \small $\mathcal{O}(h)$, \small $\mathcal{O}(h^2)$    }
  \end{semilogyaxis}
  \end{tikzpicture}
  \caption{Velocity errors for different normal approximations; $k=2$, $k_u=1$, $k_g=2$ ($\tilde \bn_h$) or $k_g=1$ ($\bn_h$).}
  \label{Error_normals}
\end{figure}
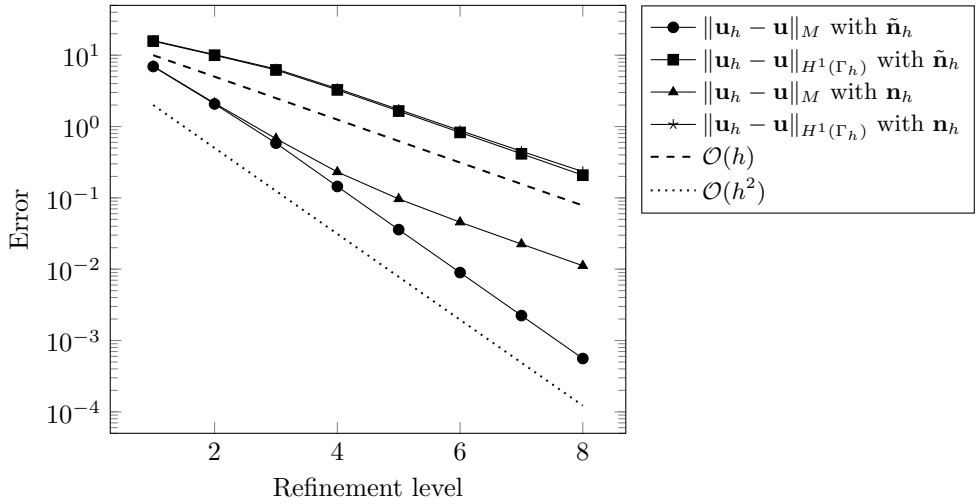

\null\newpage

\bibliographystyle{siam}
\bibliography{literatur}

\end{document}